\documentclass[a4paper,11pt]{amsart}

\usepackage{amsmath,amsthm,amsfonts,amssymb,bm,graphicx, mathabx,mathrsfs, dsfont, mathtools}

\usepackage{comment}
\usepackage
{hyperref}

\hypersetup{colorlinks=true,citecolor=blue,linkcolor=blue,urlcolor=blue,
	pdfstartview=FitH }

\theoremstyle{plain}
\newtheorem{lemma}{Lemma}
\newtheorem{theorem}[lemma]{Theorem}

\newtheorem{prop}[lemma]{Proposition}

\numberwithin{lemma}{section}
\numberwithin{equation}{section}

\mathtoolsset{showonlyrefs=true}

\theoremstyle{definition}

\theoremstyle{remark}
\newtheorem{remark}{Remark}[section]

\renewcommand{\pmod}[1]{\,(\textup{mod}\,#1)}

\newcommand{\RR}{\mathbb{R}}

\newcommand{\sumstar}{\sideset{\ }{^*}\sum}
\newcommand{\sumflat}{\sideset{}{^\flat}\sum}

\newcommand{\sumd}{\sideset{}{^d}\sum}

\newcommand{\tRe}{\textup{Re }}

\newcommand{\bfrac}[2]{\left(\frac{#1}{#2}\right)}
\newcommand{\M}{\mathcal M}

\newcommand{\W}{\mathcal W}

\newcommand{\Wt}{\widetilde{\mathcal{W}}}

 \author[J. Li]{Junxian Li}
\address{Department of Mathematics\\
	University of California, Davis\\
	One Shields Avenue\\ Davis\\ CA 95616\\ USA
}
\email{junxian@math.ucdavis.edu}

\author[X. Li]{Xiannan Li}
\address{Mathematics Department\\ Kansas State University\\ Manhattan\\ KS 66503\\ USA}
\email{xiannan@ksu.edu}

\author[Y. Lin]{Yongxiao Lin}
\address{Data Science Institute, Shandong University, Jinan 250100, China; State Key Laboratory of Cryptography and Digital Economy Security, Shandong University, Jinan 250100, China
}
\email{yongxiao.lin@sdu.edu.cn}

\title{Simultaneous non-vanishing of twists of $\mathrm{GL}_3$ and $\mathrm{GL}_1$ $L$-functions}
\begin{document}
	\maketitle
	\begin{abstract}
		We show that for a fixed $\mathrm{SL}_3(\mathbb Z)$ Hecke--Maass cusp form $F$, there are infinitely many primitive Dirichlet characters $\chi$ such that $L(1/2, F\times \chi)$ and $L(1/2, \chi)$ do not vanish \emph{simultaneously}. More precisely, we obtain an asympotic formula with a {\em power-saving} error term for the product of these central values averaged over $\{\chi\bmod q: q\in \mathcal Q\}$ where $\mathcal Q$ is a set of squarefree integers with a small divisor of suitable size.  
	\end{abstract}
	\section{Introduction}
	 The study of non-vanishing of central values of $L$-functions has attracted considerable attention in recent years, motivated partially by its deep arithmetic implications such as the Birch–Swinnerton-Dyer conjecture and its connections to Landau-Siegel zeros (see \cite{Iwaniec–Sarnak}). In this paper, we are interested in the family of $L$-functions attached to Dirichlet character twists of a fixed $\mathrm{GL}_m$ cusp form.  
	 
	 To establish non-vanishing results, one wishes to establish asymptotic formulas for moments of these families of $L$-values. Despite substantial progress for the cases when $m\leq 2$ in recent years \cite{Iwaniec–Sarnak, Bui, KhanNgo,BFKMM, Zach}, relatively few results are known for twists of higher degree $L$-functions in this family.  The problem becomes more difficult as $m$ increases and typically requires a careful choice of the family to overcome certain obstacles for higher rank groups. 
	 For $m=3$, Luo \cite{Luo} chose the family of Dirichlet characters to consist of primitive characters $\chi \bmod q$ where $q$ also varies over a set consisting products of two primes of suitable sizes. The bilinear structure of the moduli was crucially used to overcome our lack of non-trivial estimates for sums of hyper-Kloosterman sums over the moduli. Blomer--Li recently bypassed this difficulty in \cite{BLi} to obtain the asymptotic formula for the full family where $q$ is allowed to range over all integers up to $Q$. Both \cite{Luo} and \cite{BLi} obtained asymptotic formula with power-saving error terms, which serves as a first step in establishing a positive proportion of non-vanishing in the family. 
	 For $m=4$, Radziwi\l\l--Yang \cite{RY} generalized the work of Luo using more special moduli 
	 $q$ whose prime factors lie in various meticulously chosen ranges. They obtained an asymptotic formula with a logarithmic-saving error term, however, obtaining such formulae in full generality with a power-saving error term remains a challenging problem. 
	 
	 In this paper, we consider a speical case when $m=4$ where the $L$-function factors into a product of $\mathrm{GL}_3$ and $\mathrm{GL}_1$ $L$-functions and we prove an asymptotic formula with a power-saving error term for the family of $L$-values. 
To be more precise, we consider the set 
$$\mathcal Q= \mathcal Q(Q_1, Q_2) = \left\{q: \mu^2(q)=1, q=q_1q_2, Q_1/2\le q_1 \le 2Q_1, Q_2\leq q_2 \leq 2Q_2\right\}$$ for some
parameters $Q_1, Q_2 > 0$. 
We study the moment 
\begin{align}
\sum_{q\in \mathcal Q}\Psi(\frac{q}{Q})\sumflat_{\chi \bmod q}L(1/2, F\times \chi)\overline{L(1/2, \chi)}
\end{align}
where $F$ is a fixed $\mathrm{SL}_3(\mathbb Z)$ Hecke--Maass cusp form, $Q = Q_1Q_2$, $\Psi$ is a smooth compactly support function, and the sum $\sum^\flat$ means a sum over all even primitive characters. 
We obtain the following result. 
\begin{theorem}\label{MainThm}
	Let $\Psi$ be a smooth function compactly supported in $[1,2]$ and $F, \mathcal Q$ and $Q$ be as above.  Then for any $\varepsilon>0$ we have
\begin{align}
&\sum_{q\in \mathcal Q}\Psi(\frac{q}{Q})\sumflat_{\chi \bmod q}L(1/2, F\times \chi)\overline{L(1/2, \chi)}\\=&\sum_{q\in \mathcal Q}\Psi\Big(\frac{q}{Q}\Big)\phi^{\flat}(q)L(1, F\times \chi_0)+ O\Big(Q^{15/8+\varepsilon}Q_2^{9/16}+\frac{Q^{2+\varepsilon}}{Q_2^{1/4}}\Big).
\end{align}
where $\chi_0$ is the principal character modulo $q$ and $\phi^\flat(q)$ is the number of even primitive characters modulo $q$.
\end{theorem}
	Theorem \ref{MainThm} gives an asymptotic formula for the family of $L$-values when $Q^\varepsilon \ll Q_2\ll Q^{2/9-\varepsilon}$. The case for odd primitive characters can be obtained in the same way. 
As a consequence, we see that for a fixed $\mathrm{SL}_3(\mathbb Z)$ Hecke--Maass cusp form $F$, there are infinitely many primitive Dirichlet characters $\chi$ such that $L(1/2, F\times \chi)$ and $L(1/2, \chi)$ do not vanish simultaneously. 
With more work, it is plausible to establish simultaneous extreme values of $L(1/2, F\times \chi)$ and $L(1/2, \chi)$ using the resonance method developed by Soundararajan \cite{SoundExtreme} and the recent modifications by Heap--Li \cite{HeapLi}. 

We mention that the complex conjugation of the 
$L$-values in Theorem \ref{MainThm} leads to simpler expressions for the root numbers and is ultimately responsible for the power-saving error term, compared with the logarithmic-saving obtained in \cite{RY}. Similar phenomenon occurs when $F$ is replaced by a cusp form for $\mathrm{SL}_2(\mathbb Z)$ from work of Liu \cite{Liu} and Das--Khan \cite{DasKhan}. 

To prove a positive proportion of non-vanishing for the family considered in Theorem \ref{MainThm}, one still needs to obtain an asymptotic formula with a power-saving error term for the second moment of the product of the $L$-values, which continues to be an intriguing problem. This question may be compared with the eighth moment of Dirichlet $L$-functions where a logarithmic-saving has only recently been obtained by Chandee--Li--Matom{\"a}ki--Radziwi{\l}{\l} \cite{CLMR8th} with an additional $t$-averaging. 
If $F$ is replaced by a cusp form for $\mathrm{SL}_2( \mathbb Z)$, the asymptotic formula with a power-saving error term is available from the recent work of Chen--Lin \cite{ChenLin}, building on ideas from Chandee--Li--Matom{\"a}ki--Radziwi{\l}{\l} \cite{CLMR6th} and Iwaniec \cite{IwaniecEigen} (see also \cite{Luo}). 

\vspace{2ex}
\textbf{Notation.} 
We write $\ell \sim L$ to mean $L \le \ell \le 2L$.  We use $\varepsilon$ to denote an arbitrary small positive number which may differ from line to line. 

\section{Proof sketch}
We give a sketch of the proof which uses a combination of analytic, automorphic and algebraic techniques, in particular the asymptotic large sieve, the spectral theory of automorphic forms and exponential sum estimates from algebraic geometry. 

We begin with an unbalanced functional equation for $L(1/2, F\times \chi)$ to write \begin{equation}\label{GL3AFE}
	L(1/2, F\times \chi)\approx\sum_{n\leq q^{2-\delta}}\frac{A_F(1,n)\chi(n)}{\sqrt{n}}+\epsilon_\chi^3\sum_{n\leq q^{1+\delta}}\frac{A_F(n,1)\bar{\chi}(n)}{\sqrt{n}},
\end{equation}
for some small $\delta>0$ to be chosen later. Here $\epsilon_\chi$ is a normalized Gauss sum. Then we have
\begin{align}
\sum_{q\in \mathcal Q}\,\sumflat_{\chi \bmod q}L(1/2, F\times \chi)\overline{L(1/2, \chi)}=:P_1+P_2
\end{align}
where $P_1$ and $P_2$ correspond to the two terms in \eqref{GL3AFE}. 

For $P_1$, we write $\overline{L(1/2, \chi)} \approx \sum_{m\leq q}\frac{\bar{\chi}(m)}{\sqrt{m}}$ so that 
\begin{align}
P_1 \approx \sum_{q\in \mathcal Q}q \sum_{m\leq q}\frac{1}{\sqrt{m}}\sum_{\substack{n\leq q^{2-\delta}\\ n\equiv m \bmod q}}\frac{A_F(1,n)}{\sqrt{n}}.
\end{align}
Since both $m$ and $n$ are small, the complementary divisor $hq_2 := \frac{m\pm n}{q_1} \ll Q^{1-\delta} Q_2$, so it makes sense to switch from congruence conditions mod $q$ to congruence conditions mod $hq_2$ (\`{a} la  \emph{asymptotic large sieve} in \cite{CIS}) provided that $Q_2 < Q^{\delta - \varepsilon}$, say.  Then $P_1$ can be estimated by the large sieve inequalities for Dirichlet characters together with functional equations. 

For $P_2$, we write $\overline{L(1/2, \chi)} 
\approx {\epsilon}_{\bar{\chi}}\sum_{m\leq q}\frac{\chi(m)}{\sqrt{m}}$.  Noting that morally $ \epsilon_\chi^3 {\epsilon}_{\bar{\chi}}  = \epsilon_\chi^2$, and averaging over the character gives rise to Kloosterman sums, we have

\begin{align}
P_2 \approx \sum_{q\in \mathcal Q}\sum_{m\leq q}\frac{1}{\sqrt{m}}\sum_{n\leq q^{1+\delta}}\frac{A_F(n,1)}{\sqrt{n}}S(\bar{m}, \pm n; q).
\end{align} 
Using the spectral theory of automorphic forms (following Deshouillers--Iwaniec), we see that the sum of Kloosterman sums over $q_1$ becomes morally a sum over $GL(2)$ forms of level $m q_2$, so intuitively we can treat the range when $m\leq q^{c_1}$ for some $c_1<1-\delta$, since that results in a reduction in conductor.  

So far, the factorization $q = q_1 q_2$ has only made our estimates worse.  This factorization becomes crucial when dealing with the remaining range $m\sim M$ where $q^{c_1}<M\leq q$. We apply Poisson summation in $m$ which results in the illustrative range
\begin{align}\label{kl3bilinear}
\sqrt{Q} \sum_{q\in \mathcal Q}\sum_{n\sim q^{1+\delta}}\frac{A_F(n,1)}{\sqrt{n}}\sum_{m\sim q/M}\frac{1}{\sqrt{m}}\operatorname{Kl}_3(\pm mn;q),
\end{align}
where $\operatorname{Kl}_3(\pm mn;q)$ is the normalized hyper-Kloosterman sum, which is generically $O(q^\varepsilon)$.  Next we glue $m$ and $n$ together and apply Cauchy--Schwarz using the factorization of $q$ to see that the above is essentially
\begin{align}
	&Q^{1/2}\sum_{q_2\sim Q_2}\sum_{l \ll q^{2+\delta}/M}\frac{\alpha(l)}{\sqrt{l}}\sum_{q_1\sim Q_1}\operatorname{KL}_3(\pm l;q_1q_2)\\
	& \ll Q^{1/2+\varepsilon} \sum_{q_2\sim Q_2} \Big(\sum_{l \ll Q^{2+\delta }/M} \Big|\sum_{q_1\sim Q_1} \operatorname{KL}_3(\pm l; q_2q_1)\Big|^2 \Big)^{1/2},
\end{align}
where $\alpha(l) = \sum_{\substack{l=mn\\ n \sim q^{1+\delta}\\ m\sim q/M}} A_F(n,1)$.
%
Expanding the square, the quantity is  
\begin{equation*}
	\ll Q^{1/2+\varepsilon}\sum_{q_2\sim Q_2} \Big(\sum_{q_1, q_1'\sim Q_1}\sum_{l \ll Q^{2+\delta}/M}\operatorname{KL}_3(l, q_2q_1)\overline{\operatorname{KL}_3(l, q_2q_1')}\Big)^{1/2}
\end{equation*}
The diagonal term $q_1=q_1'$ contributes at most
\begin{align}
	Q^{1/2+\varepsilon}Q_2 (\frac {Q^{2+\delta}}{M} Q_1)^{1/2}\ll Q^{2+\varepsilon}(\frac{Q_2Q^\delta}{M})^{1/2}
\end{align}
which is acceptable if 
\begin{equation}\label{eqn:MQ_2bdd}
Q_2< Q^{c_1-\delta-2\varepsilon}.	
\end{equation}
 
For the off-diagonal terms $q_1\not=q_1'$, we apply Poisson summation to $l\ll L= q^{2+\delta}/M$ to obtain 
\begin{align}
	\sum_{l\ll L}\operatorname{KL}_3(l, q_2q_1)\overline{\operatorname{KL}_3(l, q_2q_1')}
	&= \frac{L}{q_2q_1q_1'}\sum_{l\ll q_2q_1q_1'/L}\sum_{\beta\bmod {q_2q_1q_1'}}\operatorname{KL}_3(\beta, q_2q_1)\overline{\operatorname{KL}_3(\beta, q_2q_1')} e\bfrac{\beta l}{q_1q_1' q_2}.
\end{align}
Note that $q_1q_1' q_2/L \gg 1$ by \eqref{eqn:MQ_2bdd}.  Assuming square-root cancellation in the $\beta$-sum, we obtain that the off-diagonal contribution is bounded by 
\begin{align}
	Q^{1/2+\varepsilon}Q_2 (Q_1^2 (Q_2Q_1^2)^{1/2})^{1/2}=Q^{2+\varepsilon}Q_2^{-1/4}
\end{align}
which is acceptable as long as $Q_2> Q^{4\varepsilon}$. 

%
%
%
As long as $\delta < c_1< 1-\delta$, we have some room for $Q_2.$
In practice, one can chose $\delta$ more precisely to minimize the error term.  We choose our $\mathcal Q$ to consist of squarefree integers so that there is no multiplicity in the sum over $q\in \mathcal Q$. The squarefree restriction and factorization property of $q$ introduce additional technical complications into the application of asymptotic large sieve and the spectral theory of automorphic forms. The exponential sum estimates are reduced to the prime moduli case where algebraic geometry can be applied and we appeal to the work of Bombieri--Sperber \cite{BS}, but gluing the variables before Cauchy--Schwarz requires further care. 

The rest of the paper is organized as follows. In Section \S \ref{sec:setup}, we prepare the set up of the problem and reduce the proof of Theorem \ref{MainThm} to Proposition  \ref{prop:P1} for $P_1$ and Proposition \ref{prop:P2} for $P_2$. After introducing various large sieve inequalities and a key Lemma \ref{lem:truncationcalc} in Section \S\ref{sec:prelim}, we evaluate the main term in $P_1$ in Section \S \ref{sec:P1} and prove Proposition \ref{prop:P1}, assuming Proposition \ref{prop:G} which handles the error terms and is proved in Section \S \ref{sec:G}. Finally Proposition \S \ref{prop:P2} is proved in Section \S \ref{sec:P2}. 

\section{Set up and key propositions}\label{sec:setup}
From \cite[eq (7.1.17) and Theorem 6.5.15]{G}, we have the approximate functional equation for $L(1/2, F\times \chi)$: 
\begin{align}\label{GL3AFE2}
L(1/2, F\times\chi)&=\sum_{n\geq 1}\frac{A_F(1,n)\chi(n)}{\sqrt{n}}V\Big(\frac{n}{q^{2-\delta}}\Big)+\epsilon_\chi^3\sum_{n\geq 1}\frac{A_F(n,1)\bar{\chi}(n)}{\sqrt{n}}V\Big(\frac{n}{q^{1+\delta}}\Big),
\end{align}
where 
\begin{align}\label{eqn:Vmellin}
V(y)= \frac{1}{2\pi i}\int_{(3)}y^{-s}G(s)\frac{\gamma(s+\mathfrak a+1/2)}{\gamma(1/2+\mathfrak a)}\frac{ds}{s}, \quad  \epsilon_\chi=\frac{\tau(\chi)}{i^\mathfrak a\sqrt{q}},\quad  \mathfrak a=\begin{cases}
0 & \chi(-1)=1,\\
1 & \chi(-1)=-1.
\end{cases}
\end{align}
for some entire function $G$ with exponential decay as $|s|\rightarrow\infty$ and $\gamma(s)=\prod_{i=1}^3\pi^{-s/2}\Gamma(\frac{s-\alpha_i}{2})$. 
We also have 
\begin{align}\label{GL1AFE}
\overline{L(1/2, \chi)}=L(1/2, \bar{\chi})=\sum_{m\geq 1}\frac{\bar\chi(m)}{\sqrt{m}}W\Big(\frac{m}{X\sqrt{q}}\Big)+\epsilon_{\bar\chi} \sum_{m\geq 1}\frac{\chi(m)}{\sqrt{m}}W\Big(\frac{mX}{\sqrt{q}}\Big) 
\end{align}
where 
\begin{align}\label{eqn:Wmellin}
W(y)=\frac{1}{2\pi i }\int_{(3)}y^{-s}G(s)\frac{\Gamma_\RR(s+\mathfrak a+1/2)}{\Gamma_\RR(1/2+\mathfrak a)}\frac{ds}{s}	
\end{align}
and $X>0$ is some parameter.
Further, \cite[Proposition 5.4]{iw-kow} and \cite[Lemma 1]{CIS} gives that
\begin{align}\label{eqn:VWbdd}
y^k W^{(k)}(y), y^k V^{(k)}(y) &=\delta_k + O_{k, A}( y^{1/21}) , y\to 0\notag\\
y^k W^{(k)}(y)\ll \exp(-cy), &\quad y^k V^{(k)}(y) \ll_{k}\ll \exp(-cy^{1/3})
\end{align}for all $A >0$, and all integers $k\ge 0$, and where 
\begin{equation}
\delta_k = 
\begin{cases}
1 \textup{ if } k=0\\
0 \textup{ otherwise.}
\end{cases}
\end{equation}
\begin{lemma}[Orthogonality of characters]\label{orthchar}
	We have for $(mn, q)=1$
	\begin{align}
	\sumflat_{\chi \bmod q}\chi(n)\bar{\chi}(m)=\frac{1}{2}\sum_{\pm}\sum_{d\mid (q, n\pm m)}\phi(d)\mu(q/d).
	\end{align}
\end{lemma}
\begin{proof}See \cite[Lemma 2]{CIS}.
\end{proof}

\begin{lemma}\label{epschi}
	For $(mn,q)=1$, we have 
	\begin{align}
	\sumflat_{\chi \bmod q} \bar{\chi}(n)\chi(m)\epsilon_\chi^3 \epsilon_{\bar{\chi}}=\frac{1}{2q}\sum_{\pm}\sum_{\substack{d\mid q\\ (d, q/d)=1}}\mu(q/d)\phi(d)S( \bar{m},\pm n, d) .
	\end{align}
\end{lemma}
\begin{proof}
We write
\begin{align}
	\sumflat_{\chi \bmod q}\chi(m)\bar{\chi}(n)\epsilon_\chi^3\epsilon_{\bar{\chi}}=\sumstar_{\chi \bmod q}\frac{1}{2}(1+\chi(-1))\chi(m)\bar{\chi}(n)\epsilon_\chi^3\epsilon_{\bar{\chi}}=\frac{1}{2}\sum_{\pm}f(\pm m\bar{n};q)
\end{align}
where 
\begin{align}
f(r;q)=\frac{1}{q}\sideset{\ }{^*}{\sum}_{\chi \bmod q}\chi(r)\tau(\chi)^2\chi(-1).
\end{align}
Note that if $(q_1,q_2)=1$, then $f(r;q_1q_2)=f(rq_2^2;q_1)f(rq_1^2;q_2)$ since
\begin{align}
f(r;q_1q_2)&=\frac{1}{q_1q_2}\sum_{\chi_1\bmod {q_1}}^*\sum_{\chi_2\bmod {q_2}}^*\chi_1\chi_2(r)\tau(\chi_1\chi_2)^2\chi_1\chi_2(-1)\\
&=\frac{1}{q_1q_2}\sum_{\chi_1\bmod {q_1}}^*\sum_{\chi_2\bmod {q_2}}^*\chi_1\chi_2(r)\chi_1(q_2^2)\chi_2(q_1^2)\tau(\chi_1)^2\tau(\chi_2)^2\chi_1\chi_2(-1)\\
&=f(rq_2^2;q_1)f(rq_1^2;q_2).
\end{align}
The evaluation of $f$ over prime power moduli is 
\begin{align}
f(r;p^k)&=\frac{1}{p^k}\sum_{s, t\bmod {p^k}}^* e_{p^k}(s-t)\sum_{\chi\bmod {p^k}}^* \chi(rst)\\
&=\frac{1}{p^k}\sum_{s,t \bmod {p^k}}^*e_{p^k}(s-t)\sum_{d\mid {(p^k, rst-1)}}\phi(d)\mu(p^k/d)\\
&=\frac{1}{p^k}\sum_{d\mid {p^k}}\phi(d)\mu(p^k/d)\sum_{\substack{s,t \bmod{p^k}\\ rst\equiv 1 \bmod {d}}}^* e_{p^k}(s-t)\\
&=\frac{1}{p^k}\sum_{d\mid {p^k}}\phi(d)\mu(p^k/d)\sum_{s\bmod{p^k}}^*\big(\sum_{x \bmod {p^k/d}}e_{p^k}(s-(1+xd)\overline{sr})-e_{p^k}(s)\delta(d=1)\big)\\
&=\frac{1}{p^k}\phi(p^k)\sum_{s\bmod {p^k}}^*e_{p^k}(s-\overline{sr})-\frac{1}{p^k}\mu(p^k)\sum_{s\bmod p^k}^*e_{p^k}(s)\\
&=\frac{1}{p^k}\sum_{\substack{d\mid p^k \\ (d, p^k/d)=1}}\mu(p^k/d)\phi(d)S(\bar{r},-1, d).
\end{align}
Thus 
\begin{align}
	f(r;q)=\frac{1}{q}\sum_{\substack{d\mid q\\ (d, q/d)=1}}\mu\big(\frac{q}{d}\big)\phi(d)S(\bar{r},-1,d)
\end{align}
which completes the proof. 
\end{proof}

Using \eqref{GL3AFE2} and \eqref{GL1AFE} with $X=q^{1/2+\eta}$ or $X=q^{-1/2-\eta}$ together with Lemma \ref{orthchar} and Lemma \ref{epschi}, we obtain 
\begin{align}
\sum_{q\in \mathcal Q}\Psi\Big(\frac{q}{Q}\Big) \sumflat_{\chi \bmod q}L(1/2, F\times \chi)L(1/2, \bar{\chi})=P_1+P_2+O\left(q^{-A}\right),
\end{align}
where 
\begin{equation}\label{P1def}
\begin{split}
P_1&=\frac{1}{2}\sum_{q\in \mathcal Q}\Psi\Big(\frac{q}{Q}\Big) \sum_{r\mid q}\phi(r)\mu(\frac{q}{r})\sum_{(m,q)=1}\frac{1}{\sqrt{m}}W\Big(\frac{m}{{q}^{1+\eta}}\Big)\sum_{\pm}\sum_{\substack{n\equiv \pm m \bmod r\\ (n, q)=1}}\frac{A_F(1,n)}{\sqrt{n}}V\Big(\frac{n}{q^{2-\delta}}\Big)\\
&=\frac{1}{2}\sum_{\substack{r, d\\ rd\in \mathcal Q}}\Psi\Big(\frac{rd}{Q}\Big) \mu(d)\phi(r)\sum_{(m, rd)=1}\frac{1}{\sqrt{m}}W\Big(\frac{m}{q^{1+\eta}}\Big)\sum_{\pm}\sum_{\substack{n\equiv \pm m\bmod r\\ (n, rd)=1}}\frac{A_F(1,n)}{\sqrt{n}}V\Big(\frac{n}{(rd)^{2-\delta}}\Big),
\end{split}
\end{equation}
and
\begin{align}\label{P2def}P_2&=\frac{1}{2}\sum_{q\in \mathcal Q}\Psi\Big(\frac{q}{Q}\Big) \sum_{(m,q)=1}\frac{1}{\sqrt{m}}W\Big(\frac{m}{q^{1+\eta}}\Big)\sum_{(n, q)=1}\frac{A_F(n,1)}{\sqrt{n}}V\Big(\frac{n}{q^{1+\delta}}\Big)\\&\quad \quad \times \frac{1}{q}\sum_{\pm}\sum_{\substack{d\mid q\\ (d,q/d)=1}}\mu\big(\frac{q}{d}\big)\phi(d)S(\bar{m},\pm n; d).
\end{align}

Our main Theorem \ref{MainThm} now follows from the following Propositions.

\begin{prop}\label{prop:P1}
	We have for any $\varepsilon>0$ that
	$$P_1 =\sum_{q\in \mathcal Q}\Psi\Big(\frac{q}{Q}\Big)\phi^{\flat}(q)L(1, F\times \chi_0)+O(Q^{2-\delta/2+\varepsilon}Q_2^{1/2} + Q^{7/4+\varepsilon} Q_2^{3/4}).
	$$
\end{prop}

\begin{prop}\label{prop:P2}
For any $\varepsilon>0$ and $\theta=\frac{7}{64}$ we have
\begin{equation}
	\begin{split}
			P_2(M,N)\ll & Q^\varepsilon \Big(Q^2Q_2^{-1/4}+Q^{7/4+\delta/2}Q_2^{5/8} +Q^{3/2+(2-\delta)\theta}Q_2^{5/4-3\theta/2}\Big).
	\end{split}
\end{equation}

\end{prop}

\begin{proof}[Proof of Theorem \ref{MainThm}]
Choosing $\delta = c\frac{\log Q_2}{\log Q}$, we have an error term 
\begin{align}
	Q^\varepsilon \Big( Q^2 Q_2^{-(c-1)/2}+Q^{7/4}Q_2^{3/4}&+Q^2Q_2^{-1/4}+Q^{7/4}Q_2^{5/8+c/2}+Q^{3/2+2\theta}Q_2^{5/4-3\theta/2-c\theta} \Big).
\end{align}

Setting the first term equal the fourth term $$ Q^2Q_2^{-(c-1)/2}=Q^{7/4}Q_2^{5/8+c/2},$$ we find $c=-\frac{1}{8}+\frac{\log Q}{4\log Q_2}$ and thus obtain an error term dominated by the first and the third term
\begin{align}
Q^\varepsilon\Big(Q^{15/8}Q_2^{9/16}+\frac{Q^2}{Q_2^{1/4}}\Big)
\end{align}
which is $\ll Q^{2-\varepsilon}$
for $Q^{8\varepsilon}\ll Q_2\ll Q^{2/9-32\varepsilon/9}$. 
\end{proof}
We prove Proposition \ref{prop:P1} in Section \S \ref{sec:P1}, and Proposition \ref{prop:P2} in Section \S \ref{sec:P2}. 

\section{Preliminary}\label{sec:prelim}
Here, we gather some standard calculations which are used repeatedly.
One key tool we will use is the large sieve for Dirichlet characters.  For later use, we state a version incorporating an additional divisiblity condition below. 

\begin{lemma}[Large Sieve with divisibility]\label{LargeSieve}
	Let $a_n$ be any complex numbers supported on $n\sim N$. We have 
	\begin{align}
		\sum_{\substack{q\leq Q\\ q_0\mid q}}\frac{q}{\phi(q)} \sumstar_{\chi \bmod q} \Big|\sum_{n} a_n\chi(n)\Big|^2 \ll  (\frac{Q^2}{q_0}+N ) \|a\|^2.
	\end{align}
\end{lemma}
\begin{proof}
	The lemma can be proved following the proof of \cite[Theorem 7.13]{iw-kow}, using \cite[Theorem 7.7]{iw-kow} and the fact that $\frac{a}{q}-\frac{a'}{q'}\gg \frac{1}{[q,q']}\gg \frac{q_0}{Q^2}$. 
\end{proof}

Next we state a hybrid version of Lemma \ref{LargeSieve}.

\begin{lemma}[Hybrid Large Sieve with divisibility]\label{HybridLargeSieve}
	Let $a_n$ be any complex numbers supported on $n\sim N$. We have 
	\begin{align}
		\int_{-T}^T \sum_{\substack{q\leq Q\\ q_0\mid q}}\frac{q}{\phi(q)} \sumstar_{\chi \bmod q} \Big|\sum_{n} a_n\chi(n) n^{it}\Big|^2 \ll  (\frac{Q^2 T}{q_0}+N ) \|a\|^2.
	\end{align}
\end{lemma}

\begin{proof}
	The result follows from proofs of previous hybrid large sieve bounds. For example, it follows from applying Gallagher  \cite[Theorem 1]{Gallagher} and our Lemma \ref{LargeSieve}.
\end{proof}

We also have the following large sieve type bound, which is \cite[Exercise 5]{iw-kow} when $q_0=1$. 

\begin{lemma}\label{LargeSieve2} For any complex numbers $\alpha_m, \beta_n$, we have 
	\begin{align}
		\sum_{\substack{q_0\mid q\leq Q}}\sumstar_{a\bmod q} \Big| \sum_{m\leq M}\sum_{n\leq N}\alpha_m\beta_n e\Big(\frac{a\bar{m}m}{q}\Big)\Big|^2\ll (\frac{Q^2}{q_0}+MN)\|\alpha\|^2\|\beta\|^2.
	\end{align}
	
\end{lemma}

Finally, we fold a technical calculation into the following Lemma, which will be used twice in what follows. 

\begin{lemma}\label{lem:truncationcalc}
	Let $K \ll Q$ be a real parameter and $d, \ell \ll Q$ be integers.  Further, let $W$ and $V$ be defined as before, and $\sumflat_{\substack{k\le K \\ \ell|k}}$ denote a sum over any subset of positive integers $k\le K$ satisfying $\ell |k$.  Let $M, N $ be real parameters which are allowed to depend on $k, d$ and $\ell$ and satisfy $M,N\ll Q^r$ for some fixed $r$.  Finally, let $A \ge 1$ and let $\rho$ be a non-negative arithmetic function satisfying
	$$\rho(k) \ll \frac{\phi(k)}{A}$$
	for all $k$.  Then
	\begin{align}
		\mathcal T &:= \sumflat_{\substack{k\le K \\ \ell|k}} \rho(k) \sum_{(m, kd)=1}\frac{1}{\sqrt{m}}W\Big(\frac{m}{M}\Big)\sum_{\pm}\sum_{\substack{n\not=m\\ n\equiv \pm m\bmod k\\ (n, kd)=1}}\frac{A_F(1,n)}{\sqrt{n}}V\Big(\frac{n}{N}\Big) \\
		&\ll \frac{Q^\varepsilon K^{2}}{A \ell}  +  \frac{ Q^\varepsilon K \sqrt{M}}{A \ell} + \frac{Q^\varepsilon K^{7/4}}{A \sqrt{\ell}}.
	\end{align}	
%
\end{lemma}

\begin{proof}
	We first add back the diagonal contribution coming from the terms $m = n$.  Since $\rho(k) \ll \phi(k)/A \le k/A$, this costs
	\begin{equation}
		\ll \frac{Q^\varepsilon K}{A} \sumflat_{\substack{k\le K \\ \ell|k}} \sum_{m \ll Q^r} \frac{|A_F(1,m)|}{m} \ll  \frac{Q^\varepsilon K^2}{A\ell}.
	\end{equation}
	
	Now, expressing the congruence conditions as a sum over characters modulo $k$, we obtain
	\begin{align}
		2\mathcal T + O\left(\frac{Q^\varepsilon K^{2}}{A\ell} \right) 
		=& \sum_{\pm}\sumflat_{\substack{k\le K \\ \ell|k}} \frac{\rho(k)}{\phi(k)}\sum_{\substack{\chi \bmod k}} \sum_{(m, dk)=1}\frac{\chi(\pm m)}{\sqrt{m}}W\Big(\frac{m}{M}\Big)\\&\times\sum_{\substack{(n, dk)=1}}\frac{\overline{\chi(n)}A_F(1,n)}{\sqrt{n}}V\Big(\frac{n}{N}\Big).
	\end{align}
We write the right hand side of the above as $\mathcal T(\chi = \chi_0) + \mathcal T(\chi \neq \chi_0)$, where $\mathcal T(\chi = \chi_0)$ is the contribution of the principal characters $\chi \bmod k$ and $\mathcal T(\chi\neq \chi_0)$ is the contribution from the rest of the characters.
	
	For $\mathcal T(\chi=\chi_0)$ we have
	\begin{align}\label{eqn:bddtrivialchar}
		\sum_{(n, dk)=1}\frac{A_F(1,n)}{\sqrt{n}}V\Big(\frac{n}{N}\Big)&=\frac{1}{2\pi i}\int_{(3)} L^{(dk)}(1/2+s, F) N^{s}G(s) \frac{\gamma(s+1/2)}{\gamma(1/2)}\frac{ds}{s}\\
		& =L^{(dk)}(1/2, F)+O(\tau(dq)N^{-\varepsilon_0})\ll Q^\varepsilon,
	\end{align}
	obtained by shifting the contour to $\Re(s) = -\varepsilon_0$.  Since $\rho(k)/\phi(k) \ll 1/A$, this gives
	\begin{align}\label{Tchi=chi0}
		\mathcal T(\chi = \chi_0) \ll \frac{Q^\varepsilon}{A} \sumflat_{\substack{k\le K \\ \ell|k}} \sqrt M \ll \frac{Q^\varepsilon K \sqrt M}{A\ell}.
	\end{align}
	
For $\mathcal T(\chi\neq \chi_0)$, we write the sum over $\chi \bmod k$ as a sum over primitive characters $\psi \bmod c$ with $c \mid k$, apply a dyadic subdivision $c \sim S$ with $S$ a power of $2$ satisfying $S \ll K$, and let $\mathcal T'(S)$ denote the contribution of the dyadic interval.  It suffices to prove
	$$\mathcal T'(S) \ll \frac{Q^\varepsilon K^2}{A\ell} + \frac{Q^\varepsilon K^{7/4}}{A\sqrt \ell}.$$
	We may assume without loss of generality that $\ell \ll K$, since otherwise the sum over $k$ is empty. Using $\rho(k)/\phi(k) \ll 1/A$ and $\ell\mid k, c\mid k$, we see that 
	\begin{align}
		\mathcal T'(S)\ll \frac{Q^\varepsilon K}{A}\sum_{c\sim S}\frac{1}{[\ell, c]}\sum_{\psi \bmod c}^*  \sum_{(m, d)=1}\frac{\chi(m)}{\sqrt{m}}W\Big(\frac{m}{M}\Big)\sum_{\substack{(n, d)=1}}\frac{\overline{\chi(n)}A_F(1,n)}{\sqrt{n}}V\Big(\frac{n}{N}\Big).
	\end{align}
Using definitions of $W$ and $V$ to expand both the $m$- and $n$-sums by Mellin inversion, moving both lines of integration to $\Re(s) = \Re(z) = \varepsilon$ without passing any poles since $\psi$ is non-principal, and using the estimates
	$$|L^{(kd)}(1/2 + z, F\times \psi)| \ll (kd)^\varepsilon |L(1/2 + z, F\times \psi)|, \quad |L^{(kd)}(1/2 + z, \psi)| \ll (kd)^\varepsilon |L(1/2 + z, \psi)|$$
	for any $z$ with $\Re z > 0$, together with the exponential decay of $G$ and the Gamma factors, we obtain
	\begin{align}
		\mathcal T'(S) 
		& \ll \frac{Q^{\varepsilon} K}{A} \sum_{c \sim S }\frac{1}{[\ell, c]} \sumstar_{\psi \bmod c} \int_{(\varepsilon)} e^{-|\Im (s)|} |L(\tfrac12+s, \psi)|  |ds| \int_{(\varepsilon)} e^{-|\Im (z)|} |L(\tfrac12+z, F\times \psi)| |dz| \\
		&\ll \frac{Q^{\varepsilon} K}{A\ell} \sum_{\delta|\ell} \delta\, \mathcal T_1^{1/2} \mathcal T_2^{1/2},
	\end{align}
	where
	\begin{align}
		\mathcal T_1 := \sum_{\substack{c \sim S\\ \delta|c} }\frac{1}{c} \sumstar_{\psi \bmod c} \int_{(\varepsilon)} \exp (-|\Im (s)|) |L(1/2+s, \psi)|^2  |ds|,
	\end{align}
	and
	\begin{align}
		\mathcal T_2 := \sum_{\substack{c \sim S\\ \delta|c} }\frac{1}{c} \sumstar_{\psi \bmod c} \int_{(\varepsilon)} \exp (-|\Im (z)|) |L(1/2+z, F\times \psi)|^2  |dz|,		
	\end{align}
	upon applying Cauchy--Schwarz separately to the sums over $c$ and $\psi$, as well as to the integrals over $s$ and $z$.  Using the approximate functional equation and the large sieve inequality for Dirichlet characters (Lemma \ref{LargeSieve}), we obtain
	$$\mathcal T_1 \ll S^\varepsilon\frac{1}{S} \left( \frac{S^2}{\delta} + \sqrt{S}\right) \ll \frac{S^{1+\varepsilon}}{\delta},$$
	using the assumption that $\delta \le \ell \ll S$, and
	$$\mathcal T_2 \ll S^\varepsilon \frac{1}{S}\left(\frac{S^2}{\delta} + S^{3/2}\right)\ll S^\varepsilon(\frac{S}{\delta}+S^{1/2}).$$
	Hence
	\begin{align}
		\mathcal T'(S) &\ll \frac{Q^{\varepsilon}K}{A\ell} \sum_{\delta|\ell} \delta \left( \frac{S}{\delta} + \frac{ S^{3/4} }{\sqrt{\delta}} \right) \\
		&\ll \frac{Q^{\varepsilon}K^2}{A\ell} + \frac{Q^{\varepsilon}K^{7/4}}{A\sqrt{\ell}},
	\end{align}
	using that $S \ll K$, which suffices.
\end{proof}

\section{Proof of Proposition \ref{prop:P1}}\label{sec:P1}
Following \cite{CIS}, we write $P_1=\mathcal D+\mathcal {N}$, where $\mathcal D$ is the contribution from the diagonal terms $n=m$ and $\mathcal{N}$ is contribution from the terms $n\neq m$. 

\subsection{Diagonal term}
The term when $n=m$ contributes to the main term. 
\begin{align}
\mathcal {D}&=\frac{1}{2}\sum_{q\in \mathcal Q}\Psi\Big(\frac{q}{Q}\Big) \sum_{r\mid q}\phi(r)\mu(q/r)\sum_{\pm}\sum_{\substack{(m,q)=1\\ r|m\pm m}}\frac{A_F(1,m)}{m}W\Big(\frac{m}{q^{1+\eta}}\Big)V\Big(\frac{m}{q^{2-\delta}}\Big)\\
&=\sum_{q\in \mathcal Q}\Psi\Big(\frac{q}{Q}\Big)\phi^{\flat}(q)\sum_{(m,q)=1}\frac{A_F(1,m)}{m}W\Big(\frac{m}{q^{1+\eta}}\Big)V\Big(\frac{m}{q^{2-\delta}}\Big)
\end{align}
where we used  
\begin{align}
\frac{1}{2}\sum_{\pm}\sum_{\substack{r\mid q\\ r|m\pm m}}\phi(r)\mu(q/r)=\phi^{\flat}(q), \quad (r,m)=1
\end{align}
which is the total number of even primitive characters $\bmod \, q$.
From the definition of $W, V$, we see that 
\begin{align}
	S := &\sum_{(m, q)=1}\frac{A_F(1,m)}{m}W\Big(\frac{m}{q^{1+\eta}}\Big)V\Big(\frac{m}{q^{2-\delta}}\Big)\\
	=&\sum_{(m, q)=1} \frac{A_F(1,m)}{m}\frac{1}{(2\pi i)^2}\int_{(3)} \int_{(3)} \Big(\frac{m}{q^{1+\eta}}\Big)^{-s}\Big(\frac{m}{q^{2-\delta}}\Big)^{-z}\\& \times G(s)G(z)\frac{\gamma(z+1/2)\Gamma_{\RR}(s+1/2)}{\gamma(1/2)\Gamma_{\RR}(1/2)}\frac{ds}{s}\frac{dz}{z}\\
	=&\frac{1}{(2\pi i)^2}\int_{(3)}\int_{(3)}L(1+s+z, F\times \chi_0)G(s)G(z)\\&\times\frac{\gamma(z+1/2)\Gamma_{\RR}(s+1/2)}{\gamma(1/2)\Gamma_{\RR}(1/2)} (q^{1+\eta})^{s}(q^{2-\delta})^{z}\frac{ds}{s}\frac{dz}{z}.
\end{align}
We first shift the $z$ contour to $\tRe z = \varepsilon_0>0$ without passing any poles.  Next, we shift the $s$ contour to $\Re(s)=-1/2+\varepsilon_0$,  encountering one pole at $s=0$.  Bounding the resulting integral gives 
\begin{align}
	S = \frac{1}{(2\pi i)}\int_{(\varepsilon_0)}L(1+z, F\times \chi_0)G(z)\frac{\gamma(z+1/2)}{\gamma(1/2)}(q^{2-\delta})^{z}\frac{dz}{z}+O(q^{(-1/2+3\varepsilon_0)(1+\epsilon)}).
\end{align}
Shifting the $z$ contour to $\Re(z)=-1/2+\varepsilon_0$, we encouter one pole at $z=0$ which gives 
\begin{align}
	S = L(1, F\times \chi_0)+ O(q^{(-1/2+\varepsilon_0)(2-\delta)})+O(q^{(-1/2+3\varepsilon_0)(1+\eta)}).
\end{align}
Thus we have 
\begin{align}\label{DiagonalP1}
	\mathcal D=\sum_{q\in \mathcal Q}\Psi\Big(\frac{q}{Q}\Big)\phi^{\flat}(q)L(1, F\times \chi_0)+O(Q^{3/2+\varepsilon}) 
\end{align}
for any $\varepsilon>0$ by choosing $\varepsilon_0, \eta$ sufficiently small. 

\subsection{Off-diagonal term}
We write
\begin{align}
\mathcal N
&=\frac{1}{2}\sum_{\substack{r, d\\ rd\in \mathcal Q}}\Psi\Big(\frac{rd}{Q}\Big) \mu(d)\phi(r)\sum_{(m, rd)=1}\frac{1}{\sqrt{m}}W\Big(\frac{m}{(rd)^{1+\eta}}\Big)\sum_{\pm}\sum_{\substack{n\not=m\\ n\equiv \pm m\bmod r\\ (n, rd)=1}}\frac{A_F(1,n)}{\sqrt{n}}V\Big(\frac{n}{(rd)^{2-\delta}}\Big).
\end{align} 

\begin{prop}\label{prop:G}
With notation as above, for any real parameter $R \ge 1$
\begin{equation}
	\mathcal N \ll \frac{Q^{2+\varepsilon}}{R} + Q^{7/4+\varepsilon}Q_2^{1/2} +   Q^{2-\delta+\varepsilon} Q_2 {R}  +  Q^{7/4-3 \delta/4+\varepsilon} Q_2^{3/2}R^{3/2}.
\end{equation}	
for any $\varepsilon>0$ by choosing $\eta$ sufficiently small in terms of $\varepsilon$. 
\end{prop}
We first prove Proposition \ref{prop:P1} assuming Proposition \ref{prop:G}.
\begin{proof}[Proof of Proposition \ref{prop:P1}]
By choosing $R = Q^{\delta/2} Q_2^{-1/2}$, we see that

\begin{align}
	\mathcal N \ll Q^\varepsilon\Big(Q^{2-\delta/2}Q_2^{1/2} + Q^{7/4} Q_2^{3/4}\Big).
\end{align}
Combining this with \eqref{DiagonalP1}, we see that 
	\begin{align}
		P_1&=\mathcal D+ O(Q^{2-\delta/2+\varepsilon}Q_2^{1/2} + Q^{7/4+\varepsilon} Q_2^{3/4}),
	\end{align}
	as desired.
\end{proof}

\section{Proof of Proposition \ref{prop:G}}\label{sec:G}
\subsection{Set up}\label{setupG} For $rd \in \mathcal Q$, we write $rd = q_1 q_2$ for $\frac{Q_1}{2} \le q_1 \le 2Q_1, Q_2\leq q_2 \leq 2Q_2$.  Since $q_1q_2$ is squarefree, we may write $r = r_1 r_2$ where $r_1 = (r, q_1)$ and $r_2 = (r, q_2)$.  Further, for notational brevity, we shall always write 
\begin{equation}\label{eqn:qdef}
	q = rd = r_1r_2d.
\end{equation}
We then write

\begin{align}\label{eqn:calGdcomp}
	\mathcal N = \frac 12 \sum_{q_2 \sim Q_2} \sum_{\substack{d_2 r_2 = q_2}} \sum_{\substack{(d,r_2)=1\\ d_2|d}} \mu(d) \mathcal N(d, r_2) 
\end{align}
where
\begin{equation}
	\begin{split}
	\mathcal N(d, r_2) = \phi(r_2)&\sum_{\substack{r_1 \\ (r_1, dr_2) = 1 \\  r_1 \text{squarefree}}} \Psi\Big(\frac{r_1r_2d}{Q}\Big) \phi(r_1) \sum_{(m,q)=1}\frac{1}{\sqrt{m}}W\Big(\frac{m}{q^{1+\eta}}\Big)\\&\times\sum_{\pm}\sum_{\substack{n\not=m\\ n\equiv \pm m \bmod r_1r_2\\ (n, q)=1}}\frac{A_F(1,n)}{\sqrt{n}}V\Big(\frac{n}{q^{2-\delta}}\Big).
	\end{split}
\end{equation}
Note that the support conditions on $\Psi$ and the conditions on $r_1, r_2$ and $d$ imply that $r_1r_2 d \in \mathcal Q$.  We write $\phi(r_1)=\sum_{a\ell=r_1}\mu(a)\ell$, giving

\begin{equation}
	\begin{split}
	\mathcal N(d, r_2) = \phi(r_2)&\sum_{\substack{a \\ (a,dr_2)=1}} \mu(a) \sum_{\substack{\ell \\ (\ell, dr_2)=1 \\ a\ell \textup{ squarefree}}} \ell \, \Psi\Big(\frac{a\ell r_2d}{Q}\Big) \sum_{(m,q)=1}\frac{1}{\sqrt{m}}W\Big(\frac{m}{q^{1+\eta}}\Big)\\
	&\times\sum_{\pm}\sum_{\substack{n\not=m\\ n\equiv \pm m \bmod a\ell r_2\\ (n, q)=1}}\frac{A_F(1,n)}{\sqrt{n}}V\Big(\frac{n}{q^{2-\delta}}\Big),
	\end{split}
\end{equation}
where now $q = a\ell r_2 d$.  The condition that $a\ell$ is squarefree is equivalent to $a$ squarefree, $\ell$ squarefree and $(a,\ell)=1$.  Keeping the first and third conditions, we detect the squarefree condition on $\ell$ by M\"obius inversion, writing $\ell = c^2 \ell'$ so that
\begin{align}
	\mathcal N(d, r_2) = \phi(r_2)\sum_{\substack{a \textup{ squarefree}\\ (a,dr_2)=1}} \mu(a) \sum_{\substack{c \\ (c, adr_2)=1}} \mu(c) c^2 \sum_{\substack{\ell' \\ (\ell', adr_2)=1}} \ell' \, \Psi\Big(\frac{ac^2\ell' r_2d}{Q}\Big) \\
	\times \sum_{(m,q)=1}\frac{1}{\sqrt{m}}W\Big(\frac{m}{q^{1+\eta}}\Big)\sum_{\pm}\sum_{\substack{n\not=m\\ n\equiv \pm m \bmod ac^2\ell' r_2\\ (n, q)=1}}\frac{A_F(1,n)}{\sqrt{n}}V\Big(\frac{n}{q^{2-\delta}}\Big)
\end{align}
with $q = ac^2\ell' r_2 d$.  We now write
\begin{align}\label{eqn:calGR}
	\mathcal N =: \mathcal G + \mathcal S
\end{align}
where $\mathcal G$ is the contribution of $adc \le R$ and $\mathcal S$ is the contribution of $adc > R$.  We estimate $\mathcal S$ in Lemma \ref{lem:calS} below and $\mathcal G$ in Lemma \ref{lem:calG}
\begin{remark}
	In previous works on the sixth and eighth moments \cite{CIS, CL8th, CLMR6th, CLMR8th}, the definition of $\mathcal G$ and $\mathcal S$ is simpler and only involves $d$.  In these works, the contributions of the principal character from both $\mathcal G$ and $\mathcal S$ are very large, and need to be canceled for the final result.  In our work, the contribution of the principal character is small, since $F$ is a cusp form.  
\end{remark}
\begin{lemma}\label{lem:calS}
	With notation as above in section \ref{setupG}, we have
	$$\mathcal S \ll \frac{Q^{2+\varepsilon}}{R} + Q^{7/4+\varepsilon}Q_2^{1/2}
	$$
\end{lemma}
for any $\varepsilon>0$
\begin{lemma}\label{lem:calG}
	With the notation as above in section \ref{setupG}, we have 
	\begin{align}
		\mathcal G\ll Q^\varepsilon\Big(Q^{2-\delta}Q_2 R+Q^{7/4-3\delta/4}Q_2^{3/2} R^{3/2}+Q^{3/2+\varepsilon}\Big)
	\end{align}
	for any $\varepsilon>0$.
\end{lemma}
Note that Proposition \ref{prop:G} follows easily from Lemma \ref{lem:calS} and Lemma \ref{lem:calG}.
\subsection{Estimates for $\mathcal S$}
We consider $\mathcal S$ in this section and prove Lemma \ref{lem:calS}.
\begin{proof}[Proof of Lemma \ref{lem:calS}]
  We apply Lemma \ref{lem:truncationcalc} where the modulus is $k = ac^2\ell' r_2$ and the weight function is $\rho(k) = c^2 \ell'\phi(r_2)$, which satisfies
	$$\rho(k) = c^2\ell' \phi(r_2)\ll Q^\varepsilon \frac{\phi(k)}{a}.$$
	Hence we can apply Lemma \ref{lem:truncationcalc} to the $\ell'$-sum with $K = Q/d$, $\ell = ac^2r_2$, $M = q^{1+\eta}$, $N = q^{2-\delta}$, $q\asymp Q$ and $A =Q^\varepsilon/a$ getting 
	\begin{align}
		&\phi(r_2) c^2 \sum_{\substack{\ell' \\ (\ell', adr_2)=1}} \ell' \, \Psi\Big(\frac{ac^2\ell' r_2d}{Q}\Big) \\
		\times &\sum_{(m,q)=1}\frac{1}{\sqrt{m}}W\Big(\frac{m}{q^{1+\eta}}\Big)\sum_{\pm}\sum_{\substack{n\not=m\\ n\equiv \pm m \bmod ac^2\ell' r_2\\ (n, q)=1}}\frac{A_F(1,n)}{\sqrt{n}}V\Big(\frac{n}{q^{2-\delta}}\Big)\\
        \ll& Q^\varepsilon\Big(\frac{Q^2}{d^2 a^2c^2r_2} + \frac{Q^{3/2}}{d a^2c^2 r_2} + \frac{Q^{7/4}}{d ^{7/4}a^{3/2} c\sqrt{r_2}}\Big).
	\end{align}
	Now write $Q\gg r = dac \ge R$.  Summing over $r$ and $r_2 \ll Q_2$, we have
	\begin{align}
		\mathcal S \ll \sum_{r_2 \ll Q_2} Q^\varepsilon\left( \frac{Q^{2}}{R r_2} + \frac{Q^{7/4}}{\sqrt{r_2}}\right) \ll \frac{Q^{2+\varepsilon}}{R} + Q^{7/4+\varepsilon}Q_2^{1/2},
	\end{align}
	which completes the proof.
\end{proof}

\subsection{Estimates for $\mathcal G$}
In this section, we shall prove Lemma \ref{lem:calG} for $\mathcal G$ using the divisor switching technique from \cite{CIS}. We first set up some further notation.
We write $m = gM$, $n = gN$ with $g = (m,n)$, so that $(M,N)=1$.  Then the congruence condition $n \equiv \pm m \bmod ac^2\ell'r_2$ together with the coprimality conditions is equivalent to (with $q=ac^2\ell'r_2d$)
$$(q, g) = 1, \quad ac^2\ell' r_2 \mid N \pm M, \quad (M,N)=1 \quad\text{and } (MN, d) = 1.$$
by noting that $(ac \ell' r_2 d, MN) = (q, MN) = 1$ is implied by $(d, MN) = 1$, $(M, N) =1$ and $ac^2\ell' r_2 \mid N \pm M$.  Detecting the remaining condition $(\ell', adr_2 g) = 1$ by $\sum_{b \mid (\ell', adr_2g)} \mu(b)$ and writing $\ell' = bk$, we obtain
\begin{equation}\label{Gdef}
	\mathcal G = \frac 12 \sum_{q_2 \sim Q_2} \sum_{\substack{d_2 r_2 = q_2}} \sum_{\substack{(d,r_2)=1\\ d_2|d}}  \mu(d) \mathcal G(d, r_2) 
\end{equation}
where 
\begin{align}\label{eqn:cleC1}
	\mathcal G(d, r_2) &= \phi(r_2)\sum_{\substack{a \textup{ squarefree}\\ (a,dr_2)=1}} \mu(a) \sum_{\substack{c\le C/ad \\ (c, adr_2)=1}} \mu(c) c^2 \sum_{\substack{g \\ (g, acr_2d)=1}} \frac{1}{g} \sum_{\substack{M, N \\ (M,N)=1\\ M\neq N \\ (MN, d) = 1}} \frac{A_F(1,gN)}{\sqrt{MN}} \\
	& \quad \times \sum_{b \mid adr_2g} \mu(b)\, b   \sum_{\substack{k \\ ac^2bkr_2 \mid N \pm M}} k \, \Psi\Big(\frac{ac^2 b k r_2d}{Q}\Big) W\Big(\frac{gM}{q^{1+\eta}}\Big) V\Big(\frac{gN}{q^{2-\delta}}\Big)
\end{align}
with $q = ac^2 b k r_2 d$.
We write $|N \pm M| = ac^2bkhr_2$ and aim to replace the same over $k$ by a sum over the complementary divisor $h$. The motivation here is that we expect $h$ to be smaller than $Q$, which results in a reduction in conductor.
The sum over $k$ leads to the new congruence $N \equiv \mp M \pmod{ac^2 b h r_2}$.  Note that since $k = \frac{|N \pm M|}{ac^2 b h r_2}$, $q = ac^2bkr_2d = \frac{d\,|N \pm M|}{h}$, the sum over $k$ is now
\begin{align}\label{eqn:aftercomplementary1}
&\sum_{\substack{h>0\\ ac^2 b h r_2 \mid N\pm M}}\Big(\frac{|N \pm M|}{ac^2 b h r_2}\Big) \quad \Psi\Big(\frac{d|N\pm M|}{Q h}\Big)W\Big(\frac{gM}{(d|N\pm M|/h)^{1+\eta}}\Big)V\Big(\frac{gN}{(d|N\pm M|/h)^{2-\delta}}\Big),
\end{align}where we have replaced $h$ by $|h|$ when $h<0$.
Denote 
\begin{align}
\mathcal W^\pm(x,y;u)=u|x\pm y|\Psi(u|x\pm y|) W\Big(\frac{x}{(u|x\pm y|Q)^{1+\eta}}\Big)V\Big(\frac{y}{(u|x\pm y|Q )^{2-\delta}}\Big).
\end{align}
Applying this notation in \eqref{eqn:aftercomplementary1} and plugging back into \eqref{eqn:cleC1}, we obtain
\begin{align}\label{eqn:GcleC1}
	\mathcal G(d, r_2) &=  Q \frac{\phi(r_2)}{dr_2} \sum_{\substack{a \textup{ squarefree} \\ (a,dr_2)=1}} \frac{\mu(a)}{a} \sum_{\substack{c\le C/ad \\ (c, adr_2)=1}} \mu(c) \sum_{\substack{g \\ (g, acr_2d)=1}} \frac{1}{g} \notag \\
	& \quad \times \sum_{b \mid adr_2g} \mu(b)  \sum_{\substack{h>0}} \sum_{\substack{M, N \\ (M,N)=1 \\ ac^2 b h r_2 \mid N\pm M\\ M\neq N}} \frac{A_F(1,gN)}{\sqrt{MN}}  \mathcal W^\pm\Big(gM, gN; \frac{d}{Qgh}\Big).
\end{align}

\subsubsection{\textbf{Mellin transforms}}
Define
\begin{align}
	\widetilde{\mathcal W}_2^\pm(s_1,s_2;u)&=\int_0^\infty\int_0^\infty\mathcal W^\pm (x,y;u)x^{s_1}y^{s_2}\frac{dx}{x}\frac{dy}{y}.\\
\end{align}

We now derive some properties of this Mellin transform.  The lemma below are similar in style to \cite[Lemma 5]{CIS} with refinements found in Chandee--Li--Matomaki--Radziwill \cite{CLMR8th}.  

\begin{lemma} \label{lem:MellinXY}
	Given a positive real number $u,$ we define
	$$ \Wt^\pm_2 (s_1, s_2; u) = \int_0^\infty \int_0^\infty \W^\pm(x, y ; u) x^{s_1}y^{s_2} \frac{dx}{x} \frac{dy}{y}.$$
	Then the functions $\Wt^\pm_2(s_1, s_2 ; u)$ are analytic in the region $\tRe s_1, \tRe s_2 > 0$. We have the Mellin inversion formula
	$$ \W^\pm(x, y ; u) = \frac{1}{(2\pi i)^2} \int_{(c_1)}\int_{(c_2)} \Wt^\pm_2 (s_1, s_2 ; u) x^{-s_1} y^{-s_2} \> d s_2 \>d s_1,$$
	when $c_1$ and $c_2$ are positive. For any integers $k,k' \geq 1, l \geq 0$ and any $s_1, s_2 \in \mathbb{C}$ with $0 < \tRe s_1, \tRe s_2 \leq 100$, one has
	\begin{equation}
		\begin{split}
			\label{eq:Wt2bound}
			|\Wt^\pm_2(s_1, s_2; u)| &\ll \frac{1} {\tRe s_1 \tRe s_2 |s_1 + s_2|^l} Q^{(1+\eta)\tRe s_1}u^{-\tRe s_2}\exp(-\kappa(\frac{1}{uQ^{2-\delta}})^{1/3})\\
			& \times \min\Big(\frac{(1+uQ^{1+\eta})^{\tRe s_2+k-1}}{|s_1|^{k}}, \frac{(1+uQ^{1+\eta})^{\tRe s_2+k'-1}}{|s_2|^{k'}}\Big)
		\end{split}
	\end{equation}
	for some constant $\kappa> 0$.
\end{lemma}

To guide the reader, note that $\mathcal W^\pm (x,y;u)$ is $uQ^{1+\eta}$ inert in $x$ and $y$ (see e.g. \cite{KPY} for the definition of inert), and so the decay in $s_1$ and $s_2$ is to be expected.  The decay in $|s_1+s_2|$ optimizes our eventual dependence on $D$, and follows similar steps to the refinement in \cite{CLMR8th}. 

\begin{proof}
	Apart from~\eqref{eq:Wt2bound} the claims go back to~\cite{CIS} (and are easy to prove). 
	The proof is based on partial integration, and for this we first study the derivatives of $\W^\pm(x, y ; u)$. 
	
	To simplify the notation, we write in this proof $V_1, V_2, \dotsc$ for unspecified functions $V_j \colon \mathbb{R}_+^3 \to \mathbb{C}$ that satisfy
	\begin{align}\label{Vjbound}
		x^{\nu_1}y^{\nu_2}z^{\nu_3}\frac{d^{\nu_1+\nu_2+\nu_3}}{dx^{\nu_1}d y^{\nu_2}dz^{\nu_3}}V_j(x,y,z)\ll \exp(-\kappa\frac{x}{Q^{1+\eta}}- \kappa(\frac{y}{Q^{2-\delta}})^{1/3})
	\end{align} and $W_1^\pm, W_2^\pm, \dotsc$ for unspecified functions $W_j^\pm \colon \mathbb{R}_+^3 \to \mathbb{C}$ that are of the form $W_j^\pm(x, y; u) = u|x\pm y| \Psi(u|x \pm y|) V_j(x, y; u |x\pm y|)$ for some such $V_j$. We then follow the derivative calucations in \cite[Lemma 6.3]{CLMR8th}.
	By \eqref{Vjbound} and the chain rule
	\[
	\frac{d}{dx} V(x, y; u |x\pm y|) = \frac{1}{x} V_1(x, y; u |x \pm y|) + u \cdot \frac{1}{u|x\pm y|} V_2(x, y; u |x \pm y|),
	\]
	we have that in the region $u |x \pm y| \in [1, 2]$ and for any $j \geq 0$
	\[
	\frac{d^j}{dx^j} V(x, y; u |x\pm y|) = \left(\frac{1}{x^j} + u^j\right) V_3(x, y; u |x \pm y|)
	\]
	and
	\[
	\frac{d^j}{dx^j} \W^\pm(x, y; u |x\pm y|) = \left(\frac{1}{x^j} + u^j\right) W_1(x, y; u |x \pm y|).
	\]
	Hence by partial integration $k$ times, we see that 
	\[
	|\Wt^\pm_2(s_1, s_2; u)| \ll \frac{1}{|s_1|^{k}} \left|\int_0^\infty \int_0^\infty (1+ux)^{k} W^\pm_1(x, y ; u) x^{s_1}y^{s_2} \frac{dx}{x} \frac{dy}{y}\right|,
	\]
	Now we substitute $x = w$ and $y = wz$, so that
	\begin{equation}
		\label{eq:Wt2bound1}
		|\Wt^\pm_2(s_1, s_2; u)| \ll \frac{1}{|s_1|^{k}} \left|\int_0^\infty \int_0^\infty (1+uw)^{k} W^\pm_1(w, wz ; u) w^{s_1+s_2} z^{s_2} \frac{dw}{w} \frac{dz}{z}\right|.
	\end{equation}
	
	Next we perform partial integration with respect to $w$. Note that, by \eqref{Vjbound} and the chain rule, we have
	\[
	\begin{split}
		\frac{d}{dw} V(w, wz, uw|1\pm z|) &= \frac{1}{w} V_5(w, wz, uw|1\pm z|) + z \cdot \frac{1}{wz} V_6(w, wz, uw|1\pm z|) \\
		& \qquad + u|1\pm z| \frac{1}{uw|1\pm z|}V_7(w, wz, uw|1\pm z|) \\
		&= \frac{1}{w} V_8(w, wz, uw|1\pm z|).
	\end{split}
	\]
	Note also that 
	\[
	\frac{d}{dw} (1+uw)^k = \frac{1}{w} (1+uw)^k \cdot k \frac{uw}{1+uw}.
	\]
	Using similar bounds for higher order derivatives and considering other terms as well, we see that 
	\[
	\frac{d^l}{dw^l} \left((1+uw)^{k} W^\pm_1(w, wz ; u)\right) = \frac{1}{w^l} (1+uw)^{k} W^\pm_2(w, wz ; u).
	\]
	Hence, by applying partial integration $l$ times to the right hand side of~\eqref{eq:Wt2bound1}, we obtain 
	\[
	\begin{split}
		|\Wt^\pm_2(s_1, s_2; u)| &\ll \frac{1}{|s_1|^k |s_1+s_2|^l} \Bigl| \int_0^\infty \int_0^\infty w^{s_1+s_2} z^{s_2} (1+uw)^{k} W^\pm_3(w, wz ; u) \frac{dz}{z} \frac{dw}{w}\Bigr|
	\end{split}
	\]
	Substituting back $w = x$ and $z = y/x$, we see that the above integral is
	\begin{equation}
		\label{eq:W2estint}
		\int_0^\infty \int_0^\infty x^{s_1} y^{s_2} (1+ux)^{k} W^\pm_3(x, y; u) \frac{dy}{y} \frac{dx}{x}.
	\end{equation}
	Similarly, we have 
	\begin{equation}
		\label{eq:W2estint2}
		|\Wt^\pm_2(s_1, s_2; u)| \ll \frac{1}{|s_2|^{k'}|s_1+s_2|^l}	\int_0^\infty \int_0^\infty x^{s_1} y^{s_2} (1+uy)^{k'} W^\pm_3(x, y; u) \frac{dy}{y} \frac{dx}{x}.
	\end{equation}
	Recall that $k \geq 1$ and $\tRe s_i > 0$ and that $W^\pm_3(x, y; u)$ is supported on $u|x \pm y| \in [1, 2]$ and satisfies $W^\pm_3(x, y; u) \ll \exp(-\kappa\frac{x}{Q^{1+\eta}}- \kappa(\frac{y}{Q^{2-\delta}})^{1/3})$ some $c>0$.
	
	In the part of~\eqref{eq:W2estint} with $x > 10/u$, we have $y \asymp x$ and $y$ is restricted to an interval of length $1/u$ (depending on $x$). Hence the contribution of $x > 10/u$ to~\eqref{eq:W2estint} is bounded by
	\[
	\begin{split}
		&u^{k-1} \int_{10/u}^\infty x^{\tRe s_1 + \tRe s_2} x^{k} \exp(-\kappa\frac{x}{Q^{1+\eta}}- \kappa(\frac{x}{Q^{2-\delta}})^{1/3}) \frac{dx}{x^2} \\
		& \ll u^{k-1} \frac{1}{\tRe s_1 + \tRe s_2 + k -1} \left(\frac{1}{u}  + Q^{1+\eta}\right)^{{\tRe s_1 + \tRe s_2 + k -1}} \exp(-\kappa\frac{1}{uQ^{1+\eta}}- \kappa(\frac{1}{uQ^{2-\delta}})^{1/3})\\
		& \ll  \frac{Q^{(1+\eta){(\tRe s_1 + \tRe s_2})}(1+uQ^{1+\eta})^{k-1}}{\tRe s_1 \cdot \tRe s_2 }\exp(-\kappa\frac{1}{uQ^{1+\eta}}).
	\end{split}
	\]
	The same bound applies to the contributions in \eqref{eq:W2estint2} with $y>10/u$ since the contribution is negligible unless $y\asymp x\ll Q^{1+\eta}$. 
	
	In the part of~\eqref{eq:W2estint} with $x \leq 10/u$, we have $y \leq 12/u$ and $\max\{x, y\} \geq 1/(2u)$ (since $|x \pm y| \in [1/u, 2/u]$), so the contribution of this part is bounded by
	\[
	\begin{split}
		&\int_0^{10/u} \int_0^{12/u} \mathbf{1}_{\max\{x, y\} \geq 1/(2u)} x^{\tRe s_1} y^{\tRe s_2} \left|W^\pm_3(x, y; u)\right| \frac{dx}{x} \frac{dy}{y} \\
		&\ll \frac{1}{\tRe s_1} \frac{1}{\tRe s_2} \min(\frac{1}{u},Q^{1+\eta})^{\tRe s_1}\min (\frac{1}{u}, Q^{2-\delta}) ^{ \tRe s_2} \exp(- \kappa(\frac{1}{uQ^{2-\delta}})^{1/3}) \\&\ll \frac{1}{\tRe s_1} \frac{1}{\tRe s_2}Q^{(1+\eta)\tRe s_1}u^{-\tRe s_2}\exp(- \kappa(\frac{1}{uQ^{2-\delta}})^{1/3}).
	\end{split}
	\]
	Thus we have 
	\begin{align}
		\widetilde{\mathcal W}_2^\pm (s_1, s_2;u)\ll \frac{Q^{(1+\eta)\tRe s_1}u^{-\tRe s_2}(1+uQ^{1+\eta})^{k-1+\tRe s_2}}{\tRe s_1\tRe s_2|s_1|^k|s_1+s_2|^l} \exp (-\kappa (\frac{1}{uQ^{2-\delta}})^{1/3}).
	\end{align}
	Similarly using \eqref{eq:W2estint2}, we have 
	\begin{align}
		\widetilde{\mathcal W}_2^\pm (s_1, s_2;u)\ll  \frac{Q^{(1+\eta)\tRe s_1}u^{-\tRe s_2}(1+uQ^{1+\eta})^{k'-1+\tRe s_2}}{\tRe s_1\tRe s_2|s_2|^{k'}|s_1+s_2|^l} \exp (-\kappa (\frac{1}{uQ^{2-\delta}})^{1/3}).	\end{align}
	The lemma follows from these two estimates. 
\end{proof}

\subsubsection{\textbf{Expression for $\mathcal G$ after divisor switching}}
Using \eqref{Gdef} and \eqref{eqn:GcleC1}, we see that 

\begin{align}\label{eqn:calGcleCcontribution}
	\mathcal G &= \frac{Q}{2} \sum_{d_2 \ll Q_2} \mu^2(d_2)\sum_{\substack{(r_2, d_2)=1\\r_2 \sim Q_2/d_2}} \sum_{\substack{(d, r_2)=1 \\ d_2|d}} \frac{\mu(d) \ \phi(r_2)}{dr_2}\sum_{\substack{a \textup{ squarefree}\\ (a,dr_2)=1}} \frac{\mu(a)}{a} \sum_{\substack{c\le C \\ (c, adr_2)=1}} \mu(c) \sum_{\substack{g \\ (g, acr_2d)=1}} \frac{1}{g}  \\
	& \quad \times \sum_{b \mid adr_2g} \mu(b) \sum_{\substack{h>0}} \sum_{\substack{M, N \\ (M,N)=1 \\ ac^2br_2h \mid N\pm M\\ (MN, d) = 1}} \frac{A_F(1, gN)}{\sqrt{MN}}  \mathcal W^\pm\Big(gM, gN; \frac{d}{Qgh}\Big).
\end{align}

To prepare for the next steps, we restrict the sums over $a, c$ and $d$ to dyadic interval $a\sim A$, $c \sim C$ and $d\sim D$ respectively, where we may assume $ACD \ll R$ by the definition of $\mathcal G$.  We further write $k = ar_2$ and by the coprimality conditions, we  write $b = b_d b_g b_k$, where $b_d|d$, $b_g | g$ and $b_k | k$ and we need to bound 
\begin{align}
	\mathcal G(A, C, D)&=\frac{Q}{2}\sum_{d_2\ll Q_2}\mu^2(d_2)\sum_{d_2\mid d}\frac{\mu(d)}{d}\sum_{\substack{k=ar_2\\ (k, d)=(a, r_2)=1\\r_2\sim Q_2/d_2}}\frac{\mu(a)\phi(r_2)}{ar_2}\\&\times\sum_{\substack{c\leq C\\ (c, kd)=1}}\mu(c)\sum_{\substack{g\\ (g, kcd)=1}}\frac{1}{g}
	 \sum_{b_d\mid d}\mu(b_d)\sum_{b_g\mid g}\mu(g)\sum_{b_k\mid k}\mu(b_k)\\&\times\sum_{h>0} \sum_{\substack{M, N \\ (M,N)=1 \\ kc^2b_db_gb_kh \mid N\pm M\\ (MN, d) = 1}} \frac{A_F(1, gN)}{\sqrt{MN}}  \mathcal W^\pm\Big(gM, gN; \frac{d}{Qgh}\Big).
\end{align}
We change the order of summation of $b_k$ and $k$ by replacing $b_k$ by $b$ and $k$ by $bk$ to obtain that $\mathcal G(A, C, D)$ is bounded by
\begin{equation}\label{eqn:calGcleCcontribution2}
	 \tilde{\mathcal {G}}(A, C, D) := \frac{Q^{1+\varepsilon}}{A} \sum_{\substack{d\sim D\\\mu^2(d)=1}} \frac{1}{d}\sum_{d_2\mid d}\sum_{\substack{c\sim C}} \sum_{\substack{g}} \frac{1}{g}  \sum_{\substack{b_d \mid d \\ b_g|g}} \sum_{\substack{b \ll Q_2 A\\ (b,d)=1\\\mu^2(b)=1}}S(c^2 b^2 b_d b_g, g, d,k)
\end{equation}
where

\begin{equation}\label{eqn:Sdef}
	S(\ell, b, g, d) := \sum_{\substack{h>0}}  \sum_{\substack{k\ll Q_2A/b}}\Big|\sum_{\substack{M, N \\ (M,N)=1 \\ \ell h k \mid N\pm M\\ (MN, d) = 1}} \frac{A_F(1, gN)}{\sqrt{MN}}  \mathcal W^\pm\Big(gM, gN; \frac{d}{Qgh}\Big)\Big|.
\end{equation}

\subsubsection{\textbf{Estimates for $S(\ell, b, g, d)$}}

Note that $(M, N) = 1$ and $\ell h k| N\pm M$ implies that $(\ell h k, MN) = 1$.  Thus we use $\psi \bmod \ell h k$ to detect the condition $N\equiv \mp M\bmod \ell h k$ and we write
\begin{align}
	S(\ell, b,g, d) \leq \mathcal {M}+ \mathcal {E}
\end{align}
where $\mathcal {M}$ is the contribution from the principal character
\begin{align}\label{Mdef}
\mathcal {M}:=& \sum_{\substack{h>0}} \sum_{k\ll Q_2A/b} \frac{1}{\phi(\ell h k)}\Big| \sum_{\substack{M, N \\ (M,N)=1 \\ (MN, d \ell hk) = 1}} \frac{A_F(1, gN)}{\sqrt{MN}}  \mathcal W^\pm\Big(gM, gN; \frac{d}{Qgh}\Big) \Big|
\end{align} 
and $\mathcal {E}$ is the contribution from non-principal characters
\begin{align}\label{eqn:E}
\mathcal {E}:= \sum_{\substack{h>0}}\sum_{k\ll Q_2A/b} \frac{1}{\phi(\ell h k)} \sum_{\substack{\chi \bmod \ell hk \\ \chi \neq \psi_0}} \Big|\sum_{\substack{M, N \\ (M,N)=1 \\ (MN, d) = 1}} \frac{\chi(M) \overline{\chi(\mp N)} A_F(1, gN)}{\sqrt{MN}}  \mathcal W^\pm\Big(gM, gN; \frac{d}{Qgh}\Big)\Big|.
\end{align}

We record the following technical Lemma about our generating function for the evaluation of the $M,N$-sum.

\begin{lemma}\label{lem:generatingfunction} Let $g,d\in \mathbb Z_{>0}$, $\psi $ be any Dirichlet character and
	let
	$$
	\mathcal L=\mathcal L (g, d,F, \psi, s_1, s_2)= \sum_{\substack{M, N=1\\(MN, d)=1\\(M,N)=1, M\not=N}} \frac{\psi(M)\bar{\psi}(\mp N)A_F(1, gN)}{M^{\frac{1}{2}+s_1}N^{\frac{1}{2}+s_2}}.
	$$
	Then we can write
	\begin{align}
		\mathcal L = 	\bar\psi(\pm 1)\Big(L(\frac{1}{2}+s_1, \psi)L(\frac{1}{2}+s_2, F\times \overline{\psi})\mathcal F(g,d,F, \psi, s_1, s_2)-\lambda_F(g)\Big),
\end{align}
	where 
\begin{align}
	\mathcal F(g,d,F, \psi, s_1, s_2)
	=&\sum_{\substack{(g_2,d)=1\\g_1\mid g_2\mid g}}\frac{\mu(g_1)\mu(g_2)A_F(\frac{g}{g_2}, \frac{g_2}{g_1})\bar{\psi}(g_1g_2)}{(g_1g_2)^{1/2+s_2}} \mathcal H(g_2, d, F, \psi, s_1, s_2).
\end{align}
Moreover, we have $	\mathcal L$ converges absolutely for $\tRe s_1, \tRe s_2 > 1/2$, $\mathcal F$ is analytic when $\tRe(s_1), \tRe(s_2) \ge \varepsilon$ and is bounded there by $g^\vartheta Q^\varepsilon$ for $\vartheta< 5/14$ and $\mathcal H$ is analytic in the region $\tRe s_1, \tRe s_2 \ge \varepsilon$ and uniformly bounded by $(gd)^\varepsilon$ there. 

\end{lemma}

\begin{proof}
	Using the Hecke relation that 
	\begin{align}\label{heckerelationgl3}
		A_F(mn,1)=\sum_{\substack{abc=n\\ b\mid c\mid m}}\mu(b)\mu(c)A_F( \frac{c}{b},\frac{m}{c},)A_F(1,a)
	\end{align}
	we write
	\begin{align}
		\sum_{\substack{N\\ (N, dM)=1}}\frac{\bar{\psi}(\mp N)A_F(1, gN)}{N^{1/2+s_2}}=\sum_{\substack{(g_2,dM)=1\\g_1\mid g_2\mid g}}\mu(g_1)\mu(g_2) A_F(\frac{g_2}{g_1},\frac{g}{g_2}) \sum_{\substack{N\\ (N, dM)=1}} \frac{\bar{\psi}(\mp g_1g_2 N)A_F(1,N)}{(g_1g_2N)^{1/2+s_2}}
	\end{align}
	so that 
	\begin{equation}
		\begin{split}
			\mathcal L &=\bar{\psi}(\mp 1)\Big(-A_F(1,g) +\sum_{\substack{(g_2,d)=1\\g_1\mid g_2\mid g}}\frac{\mu(g_1)\mu(g_2)A_F( \frac{g_2}{g_1},\frac{g}{g_2})\bar{\psi}(g_1g_2)}{(g_1g_2)^{1/2+s_2}}  \mathcal H(g_2, d, F, \psi, s_1, s_2)\Big)
		\end{split}
	\end{equation}
	where 
	\begin{align}
		\mathcal H(g_2, d, F, \psi, s_1, s_2)=&L(1/2+s_2, F\times \bar{\psi})\prod_{p\mid d}\prod_{j=1}^3(1-\frac{\alpha_j(p)\bar{\psi}(p)}{p^{1/2+s_2}})  
	\\ & \times\sum_{\substack{M\\ (M, dg_2)=1}}\frac{\psi(M)}{M^{1/2+s_1}}\prod_{p\mid M}\prod_{j=1}^3(1-\frac{\alpha_j(p)\bar{\psi}(p)}{p^{1/2+s_2}}).
	\end{align}
	The sum over $M$ is
\begin{align}
		&\prod_{p\nmid dg_2} \Big(1 + \prod_{j=1}^3 \left( 1-\frac{\alpha_j(p)\bar{\psi}(p)}{p^{1/2+s_2}}\right) \sum_{k\ge 1} \frac{\psi(p^k)}{p^{k(1/2+s_1)} }  \Big)\\
		=&L(1/2+s_1, \psi)\prod_{p\mid dg_2} \Big(1-\frac{\psi(p)}{p^{1/2+s_1}}\Big)\prod_{p\nmid dg_2}\Big(1 - \frac{\psi(p)}{p^{1/2+s_1}} G(p, s_2)  \Big),
\end{align}
where 
\begin{equation}
	\begin{split}
	G(p, s_2) 
	&:= 1 - \prod_{j=1}^3(1-\frac{\alpha_j(p)\bar{\psi}(p)}{p^{1/2+s_2}})\\
	 &= \frac{\overline{\psi(p)}}{p^{1/2+s_2}} \sum_{i=1}^3 \alpha_i(p) - \frac{\overline{\psi(p)^2}}{p^{1+2s_2}} \sum_{i<j} \alpha_i(p)\alpha_j(p) + \frac{\overline{\psi(p)^3}}{p^{3/2+3s_2}} \prod_{i=1}^3\alpha_i(p)
	 \end{split}
\end{equation}
It now suffices to show that
\begin{equation}\label{eqn:absconverge}
	\sum_p \frac{1}{p^{1+\varepsilon}} \Big|\sum_{i=1}^3 \alpha_i(p)\Big| + \frac{1}{p^{3/2+\varepsilon}}\Big| \sum_{i<j} \alpha_i(p)\alpha_j(p)\Big| +  \frac{1}{p^{2+\varepsilon}} \Big|\prod_{i=1}^3\alpha_i(p)\Big|
\end{equation}
converges for any $\varepsilon > 0$.  Each one of the terms above involve homogeneous symmetric polynomials in $\alpha_1(p), \alpha_2(p)$ and $\alpha_3(p)$.  The proof below essentially relies on expressing these polynomials in terms of Schur polynomials.

By Rankin-Selberg theory (see Proposition 6 of \cite{Molteni}), we see that
\begin{equation}\label{eqn:RSbdd}
\sum_p \frac{1}{p^{1+\varepsilon}} \Big| \sum_{i} \alpha_i(p) \Big|^2 \textup{,} \; \; \sum_p \frac{1}{p^{2+\varepsilon}} \Big| \sum_{i\le j} \alpha_i(p) \alpha_j(p) \Big|^2, \; \sum_p \frac{1}{p^{3+\varepsilon}} \Big| \sum_{\substack{0\le i, j, k \le 3\\ i+j+k = 3}} \alpha_1(p)^i \alpha_2(p)^j \alpha_3(p)^k \Big|^2	
\end{equation}
converge for any $\varepsilon > 0$, where the indices over $i, j$ are summed over the range $1\le i, j \le 3$.
By Cauchy-Schwarz, we have 
\begin{equation}\label{eqn:1absconverge}
\sum_p \frac{1}{p^{1+\varepsilon}} \Big|\sum_{i} \alpha_i(p) \Big| < \infty
\end{equation}
 \begin{equation}\label{eqn:2absconverge}
\sum_p \frac{1}{p^{3/2+\varepsilon}} \Big|  \sum_{i\le j} \alpha_i(p) \alpha_j(p)\Big| < \infty,
 \end{equation}
  \begin{equation}\label{eqn:3absconverge}
 	\sum_p \frac{1}{p^{2+\varepsilon}} \Big| \sum_{\substack{0\le i, j, k \le 3\\ i+j+k = 3}} \alpha_1(p)^i \alpha_2(p)^j \alpha_3(p)^k  \Big| < \infty,
 \end{equation}
for all $\varepsilon > 0$.  Note that \eqref{eqn:1absconverge} shows that the first term in \eqref{eqn:absconverge} converges for any $\varepsilon>0$.
For the second term in \eqref{eqn:absconverge}, we note that $\sum_{i<j} \alpha_i(p)\alpha_j(p) = \left(\sum_{i} \alpha_i(p)\right)^2 - \sum_{i\le j} \alpha_i(p) \alpha_j(p)$ and so it follows by the triangle inequality that
\begin{equation}\label{eqn:4absconverge}
	\sum_p \frac{1}{p^{3/2+\varepsilon}}  |\sum_{i<j} \alpha_i(p)\alpha_j(p)| \le \sum_p \frac{1}{p^{3/2+\varepsilon}}  \Big|\sum_{i\le j} \alpha_i(p) \alpha_j(p)\Big| + \sum_p \frac{1}{p^{3/2+\varepsilon}}  \Big|\sum_{i} \alpha_i(p)\Big|^2<\infty
\end{equation}
using \eqref{eqn:RSbdd} and \eqref{eqn:2absconverge}.
Finally for the third term in \eqref{eqn:absconverge}, we can write
$$-\alpha_1(p)\alpha_2(p)\alpha_3(p) = \big(\sum_i \alpha_i \big)^3 - 2\big(\sum_i \alpha_i \big) \big(\sum_{i<j} \alpha_i(p)\alpha_j(p)\big) - \sum_{\substack{0\le i, j, k\le 3\\ i+j+k = 3}} \alpha_1(p)^i \alpha_2(p)^j \alpha_3(p)^k.
$$
Using that $|\sum_i \alpha_i(p)|\ll p^{1/2}$, \eqref{eqn:RSbdd}, \eqref{eqn:4absconverge}, and \eqref{eqn:3absconverge} and the triangle inequality we see that the third term in \eqref{eqn:absconverge} also converges. 
\end{proof}

\subsubsection{\textbf{Contributions from $\mathcal E$}}
We first prove the bound for $\mathcal E$ below.

\begin{prop}\label{prop:E}
	We have \begin{align}
	\mathcal {E}\ll  Q^\varepsilon A g^\vartheta \left(Q_2 \frac{Q^{1-\delta} D}{g b}  + \left(Q_2 \frac{Q^{1-\delta} D}{gb}\right)^{3/4} \ell^{1/4}\right).
	\end{align}
	
\end{prop}


\begin{proof}
By the decay of $\mathcal W$, we may and will assume that all parameters are bounded by a power of $Q$.
From \eqref{eqn:E}, we write the sum over $\chi \bmod \ell hk$ as a sum over primitive characters $\psi \bmod e$ with $e \mid \ell hk$, apply a dyadic subdivision $h\sim H, e \sim E$ with $E$ a power of $2$ satisfying 
\begin{equation}\label{eqn:Ebdd}
	E \ll HQ_2A \ell/b,
\end{equation} and let $\mathcal E(E,H)$ denote the contribution of the dyadic interval. We have
	
\begin{align}\label{eqn:EE}
	\mathcal E(E, H) 
	&= \sum_{\substack{e\sim E}} \sum_{\substack{h\sim H\\ k\ll Q_2 A/b \\ e|\ell hk}} \frac{1}{\phi(\ell  hk)} \sumstar_{\substack{\psi \bmod e}}  |S(\psi) |
\end{align}
where
\begin{align}
S(\psi) 
&= \sum_{\substack{M, N \\ (M,N)=1 \\ (MN, d) = 1}} \frac{\psi(M) \overline{\psi(\mp N)} A_F(1, gN)}{\sqrt{MN}}  \mathcal W^\pm\Big(gM, gN; \frac{d}{Qgh}\Big) \\
&=  \int_{(1/2+\varepsilon)}\int_{(1/2+\varepsilon)} \widetilde{\mathcal W}_2\Big(s_1,s_2; \frac{d}{Qgh}\Big)
\\&\quad \times g^{-s_1-s_2}
\sum_{\substack{M, N=1\\(MN, d)=1\\(M,N)=1, M\not=N}} \frac{\psi(M)\bar{\psi}(\mp N)A_F(1, gN)}{M^{\frac{1}{2}+s_1}N^{\frac{1}{2}+s_2}} ds_1ds_2.
\end{align}
by Mellin inversion in as in Lemma \ref{lem:MellinXY}. By Lemma \ref{lem:generatingfunction}, we see that 

\begin{align}
	S(\psi) 
	&=  \int_{(\varepsilon)}\int_{(\varepsilon)} \widetilde{\mathcal W}_2\Big(s_1,s_2; \frac{d}{Qgh}\Big) g^{-s_1-s_2} \bar\psi(\pm 1)\\
	&\times \Big(L(\frac{1}{2}+s_1, \psi)L(\frac{1}{2}+s_2,F\times \overline{\psi})\mathcal F(g,d,F, \psi, s_1, s_2)-A_F(1,g)\Big)
 ds_1ds_2,
\end{align}
where $\mathcal F$ is as in Lemma \ref{lem:generatingfunction}, and we have moved the contours of integration to $\tRe s_1, \tRe s_2 = \varepsilon$ without passing any poles since $\psi$ is non-principle.

For $\tRe s_1 = \tRe s_2 = \varepsilon$, Lemma \ref{lem:MellinXY} implies that
\begin{equation}\label{eqn:mellinbdd}
	|\Wt^\pm_2(s_1, s_2; u)| \ll Q^{\varepsilon}\exp(-\kappa(\frac{1}{uQ^{2-\delta}})^{1/3})
	\Big(\frac{(1+uQ^{1+\eta})^{k-1}}{|s_1|^{k}}\Big) |s_1+s_2|^{-l},
\end{equation}for any $k, l \ge 0$ by choosing $\eta$ sufficiently small (in terms of $\varepsilon$).
Putting $z = s_1+s_2$, we see that the contribution of $z \gg Q^\varepsilon$ is bounded by $Q^{-A}$ for any $A$ by choosing $\eta$ sufficiently small, and so is negligible.  For $u = d/Qgh$, we see that the contribution of $s_1 \gg \max(1, \frac{d}{gh}) Q^\varepsilon$ is also negligible.  For ease of notation, write $S_0 = \max(1, \frac{D}{gh}) Q^\varepsilon$ and the contribution from $uQ^{2-\delta}<Q^{-\varepsilon}$ is also negligible. Thus we can assume
\begin{equation}\label{eqn:Hbdd}
	h \sim H \ll \frac{Q^{1-\delta +\varepsilon} D}{g}.
\end{equation}
Using the bound \eqref{eqn:mellinbdd} with $k=1$, we see that

\begin{align}
	S(\psi)  
	&\ll Q^\varepsilon \max_{\substack{z = 2\varepsilon + it\\ t\ll Q^\varepsilon}}\int_{\substack{(\varepsilon)\\ s \ll S_0}} \frac{1}{1+|s|}
	\big|L(\tfrac{1}{2}+s, \psi)L(\tfrac{1}{2}+z - s, F\times \overline{\psi})\mathcal F(g,d, F, \psi, s, z-s)\big|
	ds \\
	&\quad+ Q^\varepsilon |\lambda_F(g)| + Q^{-A} \\
	&\ll Q^\varepsilon g^\vartheta \max_{\substack{z = 2\varepsilon + it\\ t\ll Q^\varepsilon}}\int_{\substack{(\varepsilon)\\ s \ll S_0}} \frac{1}{1+|s|}
	\left(\Big|L(\tfrac{1}{2}+s, \psi)L(\tfrac{1}{2}+z - s, F\times \overline{\psi})\Big| + 1\right) ds\label{eqn:Spsibdd}
\end{align}
using the bound for $\mathcal F$ in Lemma \ref{lem:generatingfunction}.


Further, applying a subdivision to the integral over $s$, we have either $|s| \ll Q^\varepsilon $ or $s \sim S$ where $S$ runs through powers of $2$ with $S \ll \frac{Q^\varepsilon D}{gH}$ so that we may also assume $D \gg gH$.  
Finally, note that
$$\sum_{\substack{h\sim H \\ e|\ell h}} \sum_{k\ll Q_2A/b}\frac{1}{\phi(\ell hk)} \ll Q^\varepsilon  \frac{1}{\ell e} \sum_{\lambda|\ell} \lambda \delta(\lambda|e).
$$
Putting this into \eqref{eqn:EE} gives that the contributions from $s\sim S$ to $\mathcal E(E, H)$, denoted by $\mathcal E(E,H,S)$, can be bounded by 
\begin{align}
	\mathcal E (E, H, S) 
	&\ll \frac{  Q^\varepsilon g^\vartheta}{\ell ES} \sum_{\lambda|\ell}\lambda  \sum_{\substack{e\sim E\\ \lambda|e}}  \sumstar_{\substack{\psi \bmod e}}  \max_{\substack{z = 2\varepsilon + it\\ t\ll Q^\varepsilon}}\int_{\substack{(\varepsilon)\\ s \ll S}} \Big(\Big|L(\frac{1}{2}+s, \psi)L(\frac{1}{2}+z - s, F\times \overline{\psi})\Big| + 1\Big) ds \\
	&\ll \frac{  Q^\varepsilon g^\vartheta}{\ell ES} \sum_{\lambda|\ell}\lambda \Big((\mathcal S_1 \mathcal S_2)^{1/2} + \frac{E^2 S}{\lambda}\Big)
\end{align}
upon applying Cauchy-Schwarz, where

\begin{align}
	\mathcal S_1 
	&= \sum_{\substack{e\sim E\\ \lambda|e}}  \sumstar_{\substack{\psi \bmod e}} \int_{\substack{(\varepsilon)\\ s \ll S}} \Big|L(\frac{1}{2}+s, \psi)\Big|^2 ds \ll Q^\varepsilon \left(\frac{E^2 S}{\lambda} + ES\right),
\end{align}
by the approximate functional equation and the hybrid large sieve bound from Lemma \ref{HybridLargeSieve}, and similarly

\begin{align}
	\mathcal S_2 
	&= \sum_{\substack{e\sim E\\ \lambda|e}}  \sumstar_{\substack{\psi \bmod e}} \int_{\substack{(\varepsilon)\\ s \ll S}} \Big|L(\frac{1}{2}+z - s, F\times \overline{\psi})\Big|^2 ds \ll Q^\varepsilon \left(\frac{E^2 S}{\lambda} + (ES)^{3/2}\right).
\end{align}
Since $\lambda \ll E$ we have $ES \ll E^2 S/\lambda$.  It follows that

\begin{align}
	\mathcal E (E, H, S) 
	&\ll \frac{  Q^\varepsilon g^\vartheta}{\ell } \sum_{\lambda|\ell}\lambda \left(\frac{E}{\lambda} + \frac{E^{3/4} S^{1/4}}{\lambda^{1/2}}\right) \ll Q^\varepsilon g^\vartheta \left(\frac{E}{\ell} + \frac{E^{3/4} S^{1/4}}{\ell^{1/2}}\right).
\end{align}
We now sum $E \ll HQ_2 A \ell/b$ from \eqref{eqn:Ebdd} dyadically to get
\begin{align}
		\mathcal E (H,S):=\sumd_E \mathcal E (E; H, S) 
	&\ll Q^\varepsilon g^\vartheta \left(\frac{HQ_2 A}{b} +\Big(\frac{HQ_2 A}{b} \Big)^{3/4} \ell^{1/4} S^{1/4}\right).
\end{align}
Further summing dyadically $S \ll \max \big(1, \frac{D}{gH}\big)Q^\varepsilon$ and $H \ll \frac{Q^{1-\delta+\varepsilon} D}{g}$, we have

\begin{align}
	\mathcal E \ll  Q^\varepsilon g^\vartheta \Big(Q_2 A \frac{Q^{1-\delta} D}{g b}  + \Big(Q_2 A \frac{Q^{1-\delta} D}{gb}\Big)^{3/4} \ell^{1/4}\Big),
\end{align}
as desired.
\end{proof}

Putting the estimate from Proposition \ref{prop:E} into \eqref{eqn:calGcleCcontribution2}, we have the contribution from non-principal characters can be bounded by
\begin{align}\label{eqn:calGcleCcontribution3}
		 \tilde{\mathcal G}(A, C, D;\mathcal E) \ll &Q^{1+\varepsilon}\sum_{\substack{d\sim D}} \frac{1}{d} \sum_{\substack{c\sim C}} \sum_{\substack{g}} \frac{1}{g^{1+3/4 - \vartheta}}  \sum_{\substack{b_d \mid d \\ b_g|g}} \sum_{b \ll Q_2 A}\\
		& \times  \left(Q_2 \frac{Q^{1-\delta} D}{g b}  + \big(Q_2 \frac{Q^{1-\delta} D}{g b}\big)^{3/4} \big(\frac{b_db_g b^2 c^2 }{A}\big)^{1/4}\right)\\
		&\ll Q^{1+\varepsilon}\sum_{\substack{d\sim D}} \frac{1}{d} \sum_{\substack{c\sim C}} \sum_{\substack{g}} \frac{1}{g^{1+3/4 - \vartheta}}\\ & \times \left(Q_2 \frac{Q^{1-\delta} D}{g }  + \big(Q_2 \frac{Q^{1-\delta} D}{g }\big)^{3/4} \big(Dg c^2 \big)^{1/4}Q_2^{3/4}A^{1/2}\right)
\end{align}
using $b_d\ll D$ and $b_g\leq g$. 
Note that the sum over $g$ still converges since $\vartheta < 1/2$.  Summing over $d\sim D$, we obtain

\begin{align}\label{eqn:calGcleCcontribution4}
\tilde{	\mathcal G}(A, C, D;\mathcal E) 
	&\ll Q^{1+\varepsilon} \sum_{\substack{c\sim C}} \left(Q_2 Q^{1-\delta} D  + \left(Q_2 Q^{1-\delta} D\right)^{3/4} (D c^2)^{1/4} Q_2^{3/4} A^{1/2} \right) \\
	&\ll Q^\varepsilon\Big(Q_2 Q^{2-\delta} D C  +  Q^{\frac{7}{4}-\frac{3}{4} \delta}  Q_2^{3/2}DC^{3/2} A^{1/2}\Big).
\end{align}
Dyadically summing $A^{1/2} D C^{3/2}$ over $ACD \ll R$ gives the bound $ R^{3/2}$ appearing in Proposition \ref{prop:G}.
\subsubsection{\textbf{Contributions from $\mathcal M$}}
Now we estimate the contributions of the principal characters.
\begin{prop}\label{lem:M}
	With $\mathcal M$ as in \eqref{Mdef}, we have 
	\begin{align}
		\mathcal{M} \ll \frac{Q^{1/2+\varepsilon} \mathcal J(g)} {\ell},
	\end{align}
	where
	\begin{align}\label{eqn:calJgdef}
		\mathcal J(g)=\sum_{\substack{(g_2,d)=1\\g_1\mid g_2\mid g}}\frac{|A_F( \frac{g_2}{g_1},\frac{g}{g_2})|}{(g_1g_2)^{1/2}}.
	\end{align}
\end{prop}
\begin{proof}
	We have 
	\begin{align}
		\mathcal {M}\ll \frac{Q^\varepsilon}{\ell } \sum_{\substack{h>0}} \frac{1}{hk} \Big|\sum_{\substack{M, N \\ (M,N)=1 \\ (MN, d \ell h) = 1}} \frac{A_F(1, gN)}{\sqrt{MN}}  \mathcal W^\pm\Big(gM, gN; \frac{d}{Qgh}\Big) \Big|.
	\end{align} 
To bound the sums over $M$ and $N$, we use Lemma \ref{lem:MellinXY} to Mellin invert $\mathcal W$ to seperate $M,N$.  The key difference is that we move $\Re(s_1)=1/2+\varepsilon$ due to the pole at $s_1=1$ which gives
	\begin{align}
	\mathcal {M}\ll \frac{Q^{1/2+\varepsilon} \mathcal J(g)}{\ell} \sum_{\substack{h>0\\ k\ll A Q_2}} \frac{1}{hk} \max_{\substack{z = 2\varepsilon + it\\ t\ll Q^\varepsilon}}\int_{\substack{(\varepsilon)\\ s \ll S_0}} \frac{1}{1+|s|}
	\Big(\Big|L(1+s, \psi_0)L(\frac{1}{2}+z - s, F\times \overline{\psi_0})\Big| + 1\Big) ds
	\end{align}
	upon following the steps that led to \eqref{eqn:Spsibdd} where $\psi_0$ is the principal character modulo $\ell h$.  Hence
	
	$$\mathcal M \ll  \frac{Q^{1/2+\varepsilon} \mathcal J(g)}{\ell},
	$$
	as claimed. 
\end{proof}

Next we estimate the contribution from the $g$-sum to 

\begin{lemma}\label{lem:gsumbdd}
	With $\mathcal J(g)$ as defined in \eqref{eqn:calJgdef}, we have
	\begin{equation}
		\sum_{g\ll Q^{1+\varepsilon}} \frac{\mathcal J(g)}{g} \ll Q^{\varepsilon}
	\end{equation}
\end{lemma}
\begin{proof}
	By definition
	\begin{align}
		\sum_{g \ll Q^{1+\varepsilon}} \frac{\mathcal J(g)}{g}
		&\ll \sum_{\substack{g\ll Q^{1+\varepsilon}\\g_1\mid g_2\mid g }}\frac{|A_F( \frac{g_2}{g_1},\frac{g}{g_2})|}{(g_1g_2)^{1/2}g}
		\sum_{\substack{g_2 g \ll Q^{1+\varepsilon}\\ g_1\mid g_2}} \bfrac{g_2}{g_1}^{\vartheta} \frac{|A_F( 1,g)|}{(g_1g_2)^{1/2}g_2g},
	\end{align}
for $\vartheta =5/14$.  The result follows by Rankin-Selberg theory.
\end{proof}

Putting the bounds from Proposition \ref{lem:M} and Lemma \ref{lem:gsumbdd} into \eqref{eqn:calGcleCcontribution2}, we see that $\mathcal M$ contributes to $\tilde{\mathcal G}(A, C, D)$ a quantity $\tilde{\mathcal {G}}(A, C, D;\mathcal M)$ bounded by 

\begin{align}
	\tilde{\mathcal {G}}(A, C, D;\mathcal M)\ll \frac{Q^{3/2+\varepsilon}}{A} \sum_{\substack{d\sim D}} \frac{1}{d} \sum_{\substack{c\sim C}} \sum_{\substack{g}} \frac{\mathcal J(g)}{g}  \sum_{\substack{b_d \mid d \\ b_g|g}} \sum_{b \ll Q_2 A}\sum_{k\ll Q_2A/b} \frac{1}{c^2 b^2 b_d b_gk} \ll \frac{Q^{3/2+\varepsilon}}{A},
\end{align}and dyadically summing over $A, C$ and $D$ suffices for Proposition \ref{prop:G}.

\section{Estimates for $P_2$}\label{sec:P2}
Using that $\mathcal Q$ is a subset of squarefree integers, we can rewrite \eqref{P2def} as 
\begin{equation}\label{P2def2}
	\begin{split}
		P_2&=\frac{1}{2}\sum_{d'q\in \mathcal Q}\frac{\mu(d')}{d'}\Psi\Big(\frac{d'q}{Q}\Big) \sum_{(m,d'q)=1}\frac{1}{\sqrt{m}}W\Big(\frac{m}{(d'q)^{1+\eta}}\Big)\\&\quad \quad \times \sum_{(n, d'q)=1}\frac{A_F(n,1)}{\sqrt{n}}V\Big(\frac{n}{(d'q)^{1+\delta}}\Big)\frac{1}{q}\sum_{\pm}\phi(q)S(\bar{m},\pm n; q).
	\end{split}
\end{equation}
Using $\phi(q)=\sum_{d\mid q}\mu(d)\frac{q}{d}$ we see that \eqref{P2def2} can be rewritten as
\begin{align}
	P_2
	=&\frac{1}{2}\sum_{\pm}\sum_{d'dq\in \mathcal Q}\frac{\mu(d')}{d'}\frac{\mu(d)}{d}\Psi\Big(\frac{d'dq}{Q}\Big) \sum_{(m,d'dq)=1}\frac{1}{\sqrt{m}}W\Big(\frac{m}{(d'dq)^{1+\eta}}\Big)\\& \quad \times\sum_{(n, d'dq)=1}\frac{A_F(n,1)}{\sqrt{n}}V\Big(\frac{n}{(d'dq)^{1+\delta}}\Big)S(\bar{m},\pm n; dq).
\end{align}

We split the sum over $ m, n$ into dyadic ranges using some fixed smooth function $U$ supported on $[1/2, 2]$ so that  
\begin{equation}
	P_2\ll Q^\varepsilon \sup_{\substack{M, N\\M\ll Q^{1+\eta}\\ N\ll Q^{1+\delta}} }|P_2(M, N)|
\end{equation}
where 
\begin{equation}
	\begin{split}
		P_2(M,N)&=\frac{1}{2}\sum_{\pm}\sum_{\substack{d'dq\in \mathcal Q}}\frac{\mu(d')\mu(d)}{d'd}\Psi\Big(\frac{d'dq}{Q}\Big) \sum_{( m,d'dq)=1}\frac{1}{\sqrt{m}} W\Big(\frac{m}{(d'dq)^{1+\eta}}\Big)U\big(\frac{m}{M}\big)\\
		&\quad \times \sum_{(n,d'dq)=1}\frac{A_F(n,1)}{\sqrt{n}}V\Big(\frac{n}{(d'dq)^{1+\delta}}\Big)U\big(\frac{n}{N}\big)S(\bar{m}, \pm n; dq).
	\end{split}
\end{equation}
Using the definition of $\mathcal Q$, we see that $	P_2(M,N)$ can be rewritten as
\begin{equation}
	\begin{split}
	P_2(M,N)=&\frac{1}{2}\sum_{\pm }\sum_{\substack{q_2 \sim Q_2 \\ u'u_1u_2=q_2}} \mu^2(q_2) \sum_{(d',q_2)=1}\sum_{(d, q_2d')=1}\frac{\mu(u'd')\mu(u_2d)}{u'd'u_2d}\\ &\times\sum_{(q, q_2d'd)=1}\mu^2(q)\Psi(\frac{q_2d'dq}{Q})\sum_{(m, q_2d'dq)=1}\frac{1}{\sqrt{m}}W\Big(\frac{m}{(q_2d'dq)^{1+\eta}}\Big)U\big(\frac{m}{M}\big)\\
	&\times\sum_{(n, q_2d'dq)=1}\frac{A_F(n,1)}{\sqrt{n}}V\Big(\frac{n}{(q_2d'dq)^{1+\delta}}\Big)U\big(\frac{n}{N}\big)S(\bar{m}, \pm n;\frac{q_2}{u'}dq)\label{P2MNdef}.
	\end{split}
\end{equation}

\subsection{$M$ small}
Using M\"obius inversion to detect the squarefree and coprime conditions in $q$, we can
rewrite $P_2(M, N)$ as
\begin{equation}
	\begin{split}
		P_2(M,N)=&\frac{1}{2}\sum_{\pm }\sum_{\substack{q_2\sim Q_2\\ u'u_1u_2=q_2}}\sum_{(c, q_2dd')=1}\mu(c) \mu^2(q_2)\sum_{(d', q_2)=1}\sum_{(d, q_2d')=1}\frac{\mu(u'd')\mu(u_2d)}{u'd'u_2d}\\ &\times\sum_{f\mid q_2d'd}\mu(f)\sum_{(q_0, q_2d'dc^2f)=1}\mu(q_0) \sum_{q}\Psi(\frac{q_2d'dc^2fq_0q}{Q})\sum_{(m, q_2d'dc^2fq_0q)=1}\frac{1}{\sqrt{m}}\\ &\times W\Big(\frac{m}{(q_2d'dc^2fq_0q)^{1+\eta}}\Big)U\big(\frac{m}{M}\big)\sum_{(n, q_2d'dc^2f)=1} \frac{A(q_0n,1)}{\sqrt{q_0n}}\\& \times V\Big(\frac{q_0n}{(q_2d'dc^2fq_0q)^{1+\delta}}\Big)U\big(\frac{q_0n}{N}\big)S(\bar{m}, \pm q_0n;\frac{q_2}{u'}dc^2fq_0q).
	\end{split}
\end{equation}
For some paramter $D_0$ to be choosen later, we write 
\begin{equation}\label{P2MNdecomp}
	P_2(M, N)=P_2^+(M, N, D_0)+P_2^-(M, N, D_0)
\end{equation}
where $P_2^+(M, N, D_0)$ denotes the contributions from $d'dc^2fq_0>D_0$ and $P_2^-(M, N, D_0)$ denotes the contribution from $d'dc^2fq_0\leq D_0$. 

We consider $P_2^+(M, N, D_0)$ first. 
Opening the Kloosterman sum, using the usual integral representations of $W$ and $V$ from \eqref{eqn:Vmellin} and \eqref{eqn:Wmellin} to separate the sum over $m, n$ and the sum over $q$, and then using the large sieve inequality in Lemma \ref{LargeSieve2}, we have that 
%
\begin{align}
	P_2^+(M, N, D_0)\ll& Q^{\varepsilon} \max_{\substack{u, v \ll Q^\varepsilon \\ \tRe u, \tRe v = \frac{1}{\log Q} }} \sum_{q_2\sim Q_2}\sum_{u'\mid q_2}\sum_{\substack{d', d, c,f\\ f\mid q_2d'd\\ Q\gg d'dc^2fq_0>D_0}}\frac{1}{d'd}\sum_{\substack{q\ll \frac{Q}{u'd'}\\ \frac{q_2}{u'}dc^2fq_0\mid q}}\sum_{a\bmod q}^*\\&\quad\quad \times \Big|\sum_{\substack{(m, u'd'q)=1\\ m\asymp M}}\frac{1}{m^{1/2+u}}\sum_{\substack{(n, u'd')=1\\ n\asymp N/q_0}}\frac{A(q_0n,1)}{\sqrt{q_0} n^{1/2+v}}e\Big(\frac{a\bar{m}q_0n}{q}\Big)\Big|\\
	\ll&  Q^\varepsilon \sum_{q_2\sim Q_2}\sum_{u'\mid q_2}\sum_{\substack{D_0<d'dc^2fq_0\ll Q\\ f\mid q_2d'd}} \frac{1}{d'd}\Big(\frac{Q^2}{u'q_2d'^2dc^2fq_0}\Big)^{1/2}\\&\quad\quad\times\Big(\frac{Q^2}{u'q_2d'^2dc^2fq_0}+ MN\Big)^{1/2}\frac{1}{q_0}\sum_{g_1\mid g_2\mid q_0}\frac{|A_F(\frac{q_0}{g_2}, \frac{g_2}{g_1})|^2}{g_1g_2}\\
	\ll& \frac{Q^{2+\varepsilon}}{D_0^{1/2}}+ Q^\varepsilon \sqrt{Q_2MN}Q.\label{largesieveboundMsmall}
\end{align}
Here we used the Hecke relation \eqref{heckerelationgl3} in the first inequality to obtain that
\begin{align}
	&\sum_{n\ll N/q_0}|\frac{A_F(q_0n,1)}{\sqrt{q_0n}}|^2\\
	&\ll \sum_{n\ll N/q_0} \frac{1}{q_0n}\sum_{g_1\mid g_2\mid q_0}|A_F(\frac{q_0}{g_2}, \frac{g_2}{g_1})|^2 \sum_{g_1\mid g_2\mid q_0}|A_F(\frac{n}{g_1g_2})|^2 \\
	&\ll  Q^\varepsilon \frac{1}{q_0}\sum_{g_1\mid g_2\mid q_0}\frac{|A_F(\frac{q_0}{g_2}, \frac{g_2}{g_1})|^2}{g_1g_2}\sum_{n\ll N/q_0q_1q_2}\frac{|A_F(n,1)|^2}{n}\\
	&\ll \frac{Q^\varepsilon }{q_0}\sum_{g_1\mid g_2\mid q_0}\frac{|A_F(\frac{q_0}{g_2}, \frac{g_2}{g_1})|^2}{g_1g_2}\label{A(qn)L2norm}
\end{align}
by Rankin--Selberg, and in the second inequality that
\begin{align}
	\sum_{Q\gg q_0>D_0}\frac{1}{q_0^{3/2}}\sum_{g_1\mid g_2\mid q_0}\frac{|A_F(\frac{q_0}{g_2}, \frac{g_2}{g_1})|^2}{g_1g_2}\ll \sum_{g_1, g_2\ll Q}\frac{|A(g_1, g_2)|^2}{g_1^2g_2(g_1g_2)^{3/2}}\sum_{g_3\gg D_0/g_1g_2}\frac{1}{g_0^{3/2}}\ll \frac{1}{\sqrt{D_0}}
\end{align}
using Hecke relations and Rankin-Selberg. 
Note that \eqref{largesieveboundMsmall} is sufficient as long as $D_0$ is not too small and $M$ is not too large ($M \ll Q^{1-\delta}/Q_2$).

When $d' d c^2 fq_0$ is small, we can make use of the cancellations of the sums of Kloosterman sums to obtain a satisfactory upper bound for $P_2^-(M, N, D_0)$. 
We recall \cite[Lemma 9.1]{CLMR6th} which is a refinement of Theorem 10 from Deshouillers--Iwaniec in the exceptional spectrum.
\begin{prop}\label{DIbound}
	Let $C, M, N, R,S$ be positive numbers and let $g(c,m,n,r,s)$ be a smooth function with compact support in $[C, 2C]\times (0, \infty)^4$ such that \begin{align}
		\Big|\frac{\partial ^{\nu_1+\nu_2+\nu_3+\nu_4 +\nu_5}}{\partial c^{\nu_1}\partial m^{\nu_2}\partial n^{\nu_3}\partial r^{\nu_4}\partial s^{\nu_5}}g(c,m,n,r,s)\Big|\ll_{\nu_i} c^{-\nu_1}m^{-\nu_2}n^{-\nu_3}t^{-\nu_4}s^{-\nu_5}
	\end{align}
	for any $\nu_i\geq 0$, $1\leq i\leq 5$. Assume that 
	\begin{align}
		\frac{CS\sqrt{R}}{\sqrt{MN}}\gg 1.
	\end{align}For $a_m, b_{n,r,s}$ complex sequences we have for any $\varepsilon>0$
	\begin{equation}
		\begin{split}
		&\sum_{\substack{R< r\leq 2R\\ S<s\leq 2S\\ (r,s)=1 }}\sum_{\substack{0<m\leq M\\ 0<n\leq N}}a_mb_{n,r,s}\sum_{\substack{c\\(r, c)=1}}g(c,m,n,r,s)S(m\bar{r}, \pm n; sc)\\
		& \ll_\varepsilon (CMNRS)^\varepsilon L(C, M, N, R, S)\|\bm a\|\|\bm b\|
		\end{split}
	\end{equation}
	with 
	\begin{equation}
		\begin{split}
		L(C, M, N, R, S)\ll & CS\sqrt{R}\sqrt{RS}\Big(1+\frac{\sqrt{M}}{\sqrt{RS}}\Big)\Big(1+\frac{\sqrt{N}}{\sqrt{RS}}\Big)\\& \times \Big(1+\frac{C^2S^2R}{MN(1+\frac{RS}{M})^2(1+\frac{RS}{N})}\Big)^{\theta}
		\end{split}
	\end{equation}
	where $\theta=7/64$ is the admissible bound towards the Ramanujan conjecture for $\mathrm{GL}(2)$ cusp forms over $\mathbb Q$. 
\end{prop}

We estimate $P_2^-(M, N, D_0)$ by considering the contributions from each $d'\ll D_0$ and $dc^2fq_0\sim D$ with $D\ll D_0/d'$, which we denote by $P_2^-(M, N, d', D)$. We apply Proposition \ref{DIbound} to $P_2^-(M,N, d', D)$ with 
\begin{align}
	{\tt R}= M, \, {\tt S}=Q_2D,\,  {\tt C}=\frac{Q}{Q_2d'D}, \, {\tt N}=N, \, {\tt M}=1
\end{align}
 with $	\|\bm a\|^2=1$ and
\begin{align}
\|\bm b\|^2&=\sum_{\substack{m,n, s}}\Big|\sum_{\substack{\frac{q_2}{u'}dc^2fq_0=s\\ (c, q_2d'd)=1\\ u'\mid q_2, f\mid q_2d'd\\q_0\mid n}}\mu(c)\mu(f)\sum_{u_1u_2=\frac{q_2}{u'}}\frac{\mu(u_2d)}{u_2d}\frac{1}{\sqrt{m}}\frac{A_F(n,1)}{\sqrt{n}}\Big|^2\\
&\ll Q^\varepsilon\sum_{q_0\ll D} \sum_{q_0\mid s\ll Q_2D}\frac{1}{q_0}\sum_{g_1\mid g_2\mid q_0}\frac{|A_F(\frac{q_0}{g_2}, \frac{g_2}{g_1})|^2}{g_1g_2}\ll Q_2DQ^\varepsilon
\end{align}
using \eqref{A(qn)L2norm}, Hecke relation and Rankin--Selberg. 
Note that 
\begin{align}
	\frac{{\tt C S\sqrt{R}}}{\tt \sqrt{MN}}=\frac{Q\sqrt{M}}{d'\sqrt{N}}\gg \frac{Q^{1/2-\delta/2}}{D_0}
\end{align}
and thus under the condition that 
\begin{align}\label{D0condition}
	\frac{Q^{1/2-\delta/2}}{D_0}\gg 1
\end{align} we can apply Proposition \ref{DIbound} to obtain
\begin{equation}
	\begin{split}
	P_2^-(M, N;d',D)&\ll Q^{\varepsilon}\frac{Q}{d'} \sqrt{M}\sqrt{MQ_2D}\Big(1+\frac{1}{\sqrt{MQ_2D}}\Big)\Big(1+\frac{\sqrt{N}}{\sqrt{MQ_2D}}\Big)\\
	&\quad \times \Big(1+\frac{(\frac{Q}{d'})^2M}{N(1+MQ_2D)^2(1+\frac{MQ_2D}{N})}\Big)^{\theta}(Q_2 D)^{1/2}\\
	& \ll \frac{Q^{1+\varepsilon}}{d'} \Big(M\sqrt{Q_2D}+\sqrt{MN}\Big) \Big(1+\frac{Q^2}{d'^2Q_2^2D^2MN}\Big)^{\theta}(Q_2 D)^{1/2}\\
	& \ll \frac{Q^{1+\varepsilon}}{d'}\Big(MQ_2D \Big(\frac{Q^2}{M}\Big)^{\theta}+ \sqrt{MNQ_2D}\Big(1+\frac{Q^2}{d'^2Q_2^2D^2MN}\Big)^{\theta}\Big).
	\end{split}
\end{equation}
The second term on right hand side is an increasing function of $N$ and $D$.
With $N\ll Q^{1+\delta}$ and $D\ll D_0$ we have that for all $D\ll D_0$
\begin{equation}
P_2^-(M, N, d', D)\ll 	\frac{Q^{1+\varepsilon}}{d'}MQ_2D_0 \Big(\frac{Q^2}{M}\Big)^{\theta}+\frac{Q^{3/2+\delta/2+\varepsilon}}{d'}\sqrt{MQ_2D_0}\Big(1+ \big(\frac{Q^{1-\delta}}{d'^2Q_2^2D_0^2M}\big)^\theta\Big).
\end{equation}
Summing over $d'\ll D_0$, we see that 
\begin{align}
	P_2^-(M,N, D_0)\ll Q^{1+2\theta+\varepsilon}M^{1-\theta}Q_2D_0 +Q^{3/2+\delta/2+\varepsilon}\sqrt{MQ_2D_0}\Big(1+ \big(\frac{Q^{1-\delta}}{Q_2^2D_0^2M}\big)^\theta\Big).
\end{align}
%
Combining with \eqref{P2MNdecomp} and \eqref{largesieveboundMsmall}, we see that under \eqref{D0condition}
\begin{equation}\label{P2Kuznetsov}
	\begin{split}
	P_{2}(M, N)\ll&   \frac{Q^{2+\varepsilon}}{D_0^{1/2}}+
	 Q^{1+2\theta+\varepsilon}M^{1-\theta}Q_2D_0 +Q^{3/2+\delta/2+\varepsilon}\sqrt{MQ_2D_0}\Big(1+ \big(\frac{Q^{1-\delta}}{Q_2^2D_0^2M}\big)^\theta
	\Big),
	\end{split}
\end{equation}
where we have neglected the term $Q^{3/2+\delta/2+\varepsilon}\sqrt{MQ_2}$ since it is dominated by the last term. We will choose $D_0$ to minimize the right hand side of \eqref{P2Kuznetsov} to give a satisfactory bound for $P_2(M, N)$ when $M$ is not too large.

\subsection{$M$ large }
When $M$ is large, we shall apply Poisson summation to obtain a satisfactory upper bound for $P_2(M, N)$. 
Recall the expression for $P_2(M,N)$ in \eqref{P2MNdef}. We use M\"obius inversion to remove $(m, u'd')=1$ so that 
\begin{align}
	P_2(M, N)=&\frac{1}{2}\sum_{\pm }\sum_{\substack{u'u_1u_2=q_2\\ q_2\sim Q_2}}\mu^2(q_2)\sum_{(d',q_2)=1}\sum_{(d, q_2d')=1}\frac{\mu(u'd')\mu(u_2d)}{u'd'u_2d}\sum_{(q, q_2d'd)=1}\mu^2(q)\\&\times \Psi(\frac{q_2d'dq}{Q})\sum_{(n, q_2d'dq)=1}\frac{A_F(n,1)}{\sqrt{n}}V\Big(\frac{n}{(q_2d'dq)^{1+\delta}}\Big)U\big(\frac{n}{N}\big)\\ &\times \sum_{h\mid d'u'}\mu(h)\sum_{(m, \frac{q_2}{u'}dq)=1}\frac{1}{\sqrt{hm}}W\Big(\frac{hm}{(q_2d'dq)^{1+\eta}}\Big)U\big(\frac{hm}{M}\big)S(\overline{hm}, \pm n;\frac{q_2}{u'}dq)
\end{align}
Applying Poisson in $m$ to modulus $\frac{q_2}{u'}dq$, we obtain
\begin{align}
	P_2(M,N)=&\frac{1}{2}\sum_{\pm }\sum_{\substack{q_2\sim Q_2\\u'u_1u_2=q_2}}\mu^2(q_2)\sum_{(d', q_2)=1}\sum_{(d, q_2d')=1}\frac{\mu(u'd')\mu(u_2d)}{u'd'u_2d}\sum_{(q, q_2dd')=1}\mu^2(q)\Psi(\frac{q_2d'dq}{Q})\\ &\times \frac{1}{\sqrt{M}}\sum_{(n, q_2d'dq)=1}\frac{A_F(n,1)}{\sqrt{n}}V(\frac{n}{(q_2d'dq)^{1+\delta}})U\big(\frac{n}{N}\big)\sum_{h\mid d'u'}\mu(h)\\
	&\times \sumstar_{\alpha \bmod \frac{q_2}{u'}dq} S(\overline{h\alpha}, \pm n;\frac{q_2}{u'}dq)\sum_{m\equiv \alpha \bmod \frac{q_2}{u'}dq}\frac{\sqrt{M}}{\sqrt{hm}}W\Big(\frac{hm}{(q_2d'dq)^{1+\eta}}\Big)U\Big(\frac{hm}{M}\Big)\\
	=&\frac{1}{2}\sum_{\pm }\sum_{\substack{q_2\sim Q_2\\u'u_1u_2=q_2}}\mu^2(q_2)\sum_{(d', q_2)=1}\sum_{(d, q_2d')=1}\frac{\mu(u'd')\mu(u_2d)}{u'd'u_2d}\sum_{(q, q_2d'd)=1}\mu^2(q)\Psi\big(\frac{q_2d'dq}{Q}\big)\\ &\times \sqrt{M}\sum_{(n, q_2d'dq)=1}\frac{A_F(n,1)}{\sqrt{n}}V\Big(\frac{n}{(q_2d'dq)^{1+\delta}}\Big)U\big(\frac{n}{N}\big)\sum_{h\mid d'u'}\frac{\mu(h)}{h}\\
	&\times \sum_{m}\operatorname{KL}_3(\mp \bar{h}mn, \frac{q_2}{u'}dq)\Omega\Big(\frac{M}{(q_2d'dq)^{1+\eta}}, \frac{mMu'}{hq_2dq}\Big)
\end{align}
where 
\begin{equation}
	\begin{split}
	\operatorname{KL}_3(m, q)&=\frac{1}{q}\sumstar_{u,v \bmod q}e\Big(\frac{mu+v+\overline{uv}}{q}\Big).
\\\Omega(A, B)
	&=\int\frac{1}{\sqrt{x}}W(Ax)U(x)e(Bx)dx.
	\end{split}
\end{equation}
Note that we have $\operatorname{KL}_3(0,q)=\mu^2(q)/q$ and Deligne's bound implies that 
\begin{align}\label{kl3mqbound}
	\mathrm{KL}_3(m,q)\ll \frac{3^{\omega(q)}}{(m,q)}, \quad \text{ if } \mu^2(q)=1.
\end{align}
We have from integration by parts that 
\begin{align}\label{Omqbound}
	\Omega(A, B)\ll \Big(1+ \frac{|B|}{A+1}\Big)^{-K}
\end{align}
for any $K>0$.
Note that $\frac{M}{q_2d'dq}\ll Q^\eta$ and $q_2d'dq\sim Q$ and so \eqref{Omqbound} shows that
the contribution from $|\frac{mMu'}{hq_2dq}|\gg Q^\varepsilon$ to $P_2(M, N)$ is negligible, by choosing $\eta$ sufficiently small. Thus it is enough to consider the contribution from $|m|\ll  \frac{Q^{1+\varepsilon}h}{u'd'M}$. 

We first consider the contribution from $m=0$, which can be bounded by
\begin{align}
	Q^\varepsilon \sum_{q_2\sim Q_2}\sum_{u'u_2\mid q_2}\sum_{d',d, q}\frac{1}{u'd'u_2d}\Psi \Big(\frac{q_2d'dq}{Q}\Big)\sqrt{M}\sqrt{N}\sum_{h\mid d'u'}\frac{1}{h}\cdot\frac{u'}{q_2dq}\ll Q^\varepsilon\sqrt{MN}.
	\end{align}
	
	Now we consider the contributions from $m\not=0$. Without loss of generality, we can assume $m>0$. From \eqref{Omqbound}, we can insert a factor of $H(\frac{mMu'}{hq_2dq})$ with negligible error, where $H$ is a smooth function satisfying $H(x)\asymp (1+x)^{-K_0}$ for some sufficiently large $K_0>0$ and $H^{(j)}(x)\ll (1+x)^{-K_0-j}$ for all $j\in \mathbb N$. After using M\"obius inversion to detect $(n, q)=1$ we obtain
	\begin{align}
		P_2(M,N)\ll &\sqrt{M}Q^\varepsilon\sup_{\substack{x\asymp 1}}\sum_{\pm}\sum_{\substack{q_2\sim Q_2\\ u'u_1u_2=q_2}}\sum_{(d',q_2)=1}\sum_{(d, q_2d')=1}\frac{\mu(u'd')\mu(u_2d)}{u'd'u_2d}\sum_{(q, q_2d'd)=1}\mu^2(q)\\
		&\times\Psi\big(\frac{q_2d'dq}{Q}\big)\sum_{g\mid q}\mu(g)\sum_{\substack{(n, q_2d'd)=1\\ g\mid n}}\frac{A_F(n,1)}{\sqrt{n}}V\Big(\frac{n}{(q_2d'dq)^{1+\delta}}\Big)G(\frac{n}{N})
		 \sum_{h\mid d'u'}\frac{\mu(h)}{h}\\& \times\sum_{m}\mathrm{KL}_3(\mp \bar{h}mn,\frac{q_2}{u'}dq)  W\Big(\frac{xM}{(q_2d'dq)^{1+\eta}}\Big)e\Big(\frac{xmMu'}{hq_2dq}\Big)H \Big(\frac{mMu'}{hq_2dq}\Big)
	\end{align}
	We separate $m,n,q$ using the definition of $V, W$ and Mellin inversion to write 
	\begin{align}
		V\Big(\frac{n}{(q_2d'dq)^{1+\delta}}\Big)&=\frac{1}{2\pi i }\int_{(\varepsilon)} G(w)\frac{\gamma(w+\mathfrak a+1/2)}{\gamma (1/2+\mathfrak a)}\Big(\frac{n}{(q_2d'dq)^{1+\delta}}\Big)^{-w}\frac{dw}{w}
	\end{align}
and
	\begin{align}
		&e\Big(\frac{xMu'm}{hq_2dq}\Big)H\Big(\frac{mMu'}{hq_2dq}\Big)
		=\frac{1}{2\pi i}\int_{\varepsilon}\widetilde{H}_x(s)\Big(\frac{mMu'}{hq_2dq}\Big)^{-s}ds
	\end{align}
	where 
	\begin{align}
		\widetilde{H}_x(s)=\int y^{s-1} e(xy)H(y)dx.
	\end{align}
	We have by integration by parts and the decay of $H$ that
	\begin{align}\label{MellinHbound}
		\widetilde{H}_x(s)\ll \frac{1}{\Re(s)}\Big(\frac{1}{1+|s|}\Big)^K,\text{ for all } K>0,
	\end{align}keeping in mind that $x\ll 1$.
	
Therefore, we obtain
\begin{align}
	P_2(M,N)\ll &\sqrt{M}Q^\varepsilon\sup_{\substack{x\asymp 1}}\sum_{\pm}\sum_{\substack{q_2\sim Q_2\\ u'u_1u_2=q_2}}\sum_{(d',q_2)=1}\sum_{(d, q_2d')=1}\frac{\mu(u'd')\mu(u_2d)}{u'd'u_2d}\sum_{(q, q_2d'd)=1}\mu^2(q)\\
	&\times \Psi\big(\frac{q_2d'dq}{Q}\big) \sum_{g\mid q}\mu(g)\sum_{\substack{(n, q_2d'd)=1\\ g\mid n}}\frac{A_F(n,1)}{\sqrt{n}}G(\frac{n}{N})\int_{(\varepsilon)}G(w)\frac{\gamma(w+\mathfrak a+1/2)}{\gamma(1/2+\mathfrak a)}\\
	&\times n^{-w}(q_2d'dq)^{(1+\delta)w} \frac{dw}{w} \sum_{h\mid d'u'}\frac{\mu(h)}{h}\sum_{m\ll \frac{Q^{1+\varepsilon}h}{u'd'M}}\mathrm{KL}_3(\pm \bar{h}mn,\frac{q_2}{u'}dq) \\
	&\times W\Big(\frac{xM}{(q_2d'dq)^{1+\eta}}\Big)\int_{(\varepsilon)}\widetilde{H}_x(s)\Big(\frac{mMu'}{hq_2dq}\Big)^{-s}ds.
\end{align}
%

We combine $m,n$ by writing $l=mn$ with $l\ll L:=\frac{NQ^{1+\varepsilon}h}{u'd'M}$ to obtain 
\begin{align}
	P_2(M, N)\ll& \sqrt{M}Q^\varepsilon\sup_{\substack{x\asymp 1}} \int_{w}\int_{s}\Big|G(w)\frac{\gamma(w+\mathfrak a+1/2)}{\gamma(1/2+\mathfrak a)}\widetilde{H}_x(s)\Big|
	\\&\times \sum_{\substack{q_2\ll Q_2\\ u'u_2\mid q_2}}\sum_{(d', q_2)=1}\sum_{(d, q_2d')=1}\frac{1}{u'd'u_2d}\sum_{\substack{(g, q_2d'd)=1\\ g\leq Q/q_2d'd}}
	 \sum_{h\mid d'u'}\frac{1}{h}\\& \times \Big|\sum_{0\not=|l|\ll L}\beta(l)\sum_{(q, q_2d'dg)=1}\mathrm{KL}_3(\pm \bar{h}l, \frac{q_2}{u'}dgq)\mathscr{H}(q) \Big|dwds
\end{align}
where 
\begin{align}
	\beta(l)&=\beta_{q_2d'd,g,\frac{Q^{1+\varepsilon}h}{u'd'M},s,w}(l)=\sum_{\substack{(n, q_2d'd)=1\\ g\mid n}}\sum_{\substack{m\ll \frac{Q^{1+\varepsilon}h}{u'd'M}\\mn=l}}\frac{A_F(n,1)}{\sqrt{n}}U\big(\frac{n}{N}\big)n^{-w}m^{-s},\\
	\mathscr{H}(q)&=\mathscr{H}_{q_2d'dg, Q,x,M,\eta, s,w, \delta}(q)=\Psi\big(\frac{q_2d'dgq}{Q}\big)W\Big(\frac{xM}{(q_2d'dgq)^{1+\eta}}\Big)q^{s+(1+\delta)w}.
\end{align}
Applying Cauchy in $l$ so that the $q_2,d',d,g,h,l,q$-sum in $P_2(M, N)$ can be bounded by

\begin{align}
	&\Big(\sum_{\substack{q_2\ll Q_2\\u'u_2\mid q_2}}\sum_{(d', q_2)=1}\sum_{\substack{(d, q_2d')=1\\ q_2d'd\ll Q}}\frac{1}{u'd'u_2d}\sum_{\substack{(g, q_2d'd)=1\\ g\leq Q/q_2d'd}}
\sum_{h\mid d'u'}\frac{1}{h}\sum_{l}|\beta(l)|^2\Big)^{1/2}\\
	&\times \Big(	\sum_{\substack{q_2\ll Q_2\\ u'u_2\mid q_2}}\sum_{(d', q_2)=1}\sum_{\substack{(d, q_2d')=1\\ q_2d'd\ll Q}}\frac{1}{u'd'u_2d}\sum_{\substack{(g, q_2d'd)=1\\ g\leq Q/q_2d'd}}\sum_{h\mid d'u'}\frac{1}{h}\\&\quad\quad \times \sum_{0\not=|l|\ll L}\Big| \sum_{(q, q_2d'dg)=1}\mathrm{KL}_3( \bar{h}l,\frac{q_2}{u'}dgq)\mathscr H(q)\Big|^2\Big)^{1/2}.
\end{align}


For the first factor, we use \eqref{A(qn)L2norm}, the Hecke relations and the Rankin--Selberg bounds on coefficients to see that
\begin{align}
	&\sum_{\substack{q_2\ll Q_2\\ u'u_2\mid q_2}}\sum_{\substack{d',d\\ 
	q_2d'd\ll Q}}\frac{1}{u'd'u_2d}\sum_{\substack{(g, q_2d'd)=1\\ g\leq Q/q_2d'd}}\sum_{h\mid d'u'}\frac{1}{h}\sum_{l}|\beta(l)|^2\\
&\ll Q^\varepsilon \sum_{\substack{q_2\ll Q_2\\ u'u_2\mid q_2}}\sum_{\substack{d', d\\ q_2d'd\ll Q}}\frac{1}{u'd'u_2d}\sum_{\substack{(g, q_2d'd)=1\\ g\leq Q/q_2d'd}} \sum_{h\mid d'u'}\frac{1}{h}\sum_{\substack{g\mid n}}\frac{|A_F(n,1)|^2}{n}U\Big(\frac{n}{N}\Big)\sum_{m\ll \frac{Q^{1+\varepsilon} h}{u'd'M}}1\\
& \ll Q^\varepsilon  \frac{Q}{M}\sum_{\substack{q_2\ll Q_2\\ u'u_2\mid q_2}}\sum_{\substack{d', d\\ q_2d'd\ll Q}}\frac{1}{u'^2d'^2u_2d}\sum_{\substack{(g, q_2d'd)=1\\ g\leq Q/q_2d'd}} \frac{1}{g}\sum_{g_1\mid g_2}\frac{|A_F(\frac{g}{g_2}, \frac{g_2}{g_1})|^2}{g_1g_2}\\
& \ll Q^\varepsilon Q_2 \frac{Q}{M}.
\end{align}
For the second factor, we put a weight $\Phi(l/L)$ where  $\phi(x)$ is supported on $[-5,5]$, $\phi(x)\geq 1$ on $[-2,2]$ and $\phi^{(j)}(x)\ll (\log x)^{jC}$ for all $j$. 
Let $d^\dagger=\frac{q_2}{u'}dg$. We expand the square and apply Poisson in $l$ to modulus $d^\dagger[q, q']=d^\dagger q_0q_1q_1'$ where $q_0=(q, q'), q=q_0q_1, q'=q_0q_1'$ to obtain that for any $K>0$
\begin{align}\label{lsum}
	&\sum_{\substack{q, q'\\ (qq',q_2d'dg)=1\\\mu^2(q)=\mu^2(q')=1}}\mathscr{H}(q)\overline{\mathscr{H}(q')}\sum_{l}\Phi\Big(\frac{l}{L}\Big)\operatorname{KL}_3(\bar{h}l, d^\dagger q)\overline{\operatorname{KL}_3(\bar{h}l, d^\dagger q')}\\
	&=L\sum_{\substack{q, q'\\ (qq',q_2d'dg)=1\\\mu^2(q)=\mu^2(q')=1}}\sum_{|l|\ll \frac{d^\dagger q_0q_1q_1'Q^\varepsilon}{L}}\mathscr{H}(q)\overline{\mathscr{H}(q')}\sum_{l} \mathcal S(l, d^\dagger , q, q')\hat{\Psi}\Big(\frac{lL}{d^\dagger q_0q_1q_1'}\Big)+O_C(Q^{-K})
\end{align}
where 
\begin{align}\label{eqn:Ssumdef}
	\mathcal S(l, d,q, q')&=\frac{1}{d[q, q']}\sum_{\alpha \bmod d[q, q']}\operatorname{KL}_3(\bar{h}\alpha, dq)\overline{\operatorname{KL}_3(\bar{h}\alpha, dq')}e\Big(\frac{l\alpha}{d[q, q']}\Big).
\end{align}
Note that $\mu^2(d^\dagger q)=\mu^2(d^\dagger q')=1$ so that $d^\dagger[q, q']$ is squarefree. 


We first prove the following Lemma regarding the sum $S(l, d, q, q')$.

\begin{lemma}\label{lem:Ssumbdd}
With notation as above, we have 
$$|\mathcal S(0,d,q,q')|\le Q^\varepsilon \mathds{1}_{q=q'},
$$and for $l\neq 0$,

\begin{equation}\label{eq:main}
	|\mathcal S(l,d,q,q')|\le Q^{\varepsilon}		\frac{(dq_0,l)^{1/2}}{\sqrt{dq_0q_1q_1'}}.
\end{equation}

\end{lemma}

\begin{proof}
For $l=0$, we have 
\begin{align}
	S(0, d^\dagger, q, q')&=\frac{1}{{d^\dagger}^3q_0^3q_1^2q_1'^2}\sum_{\alpha \bmod d^\dagger q_0q_1q_1'}\sumstar_{u, v\bmod{d^\dagger q_0q_1}}\sumstar_{u', v'\bmod {d^\dagger q_0q_1'}}\\&\quad \quad e\Big(\frac{\bar{h}\alpha u+v+\overline{uv}}{d^\dagger q_0q_1}\Big)e\Big(-\frac{\bar{h}\alpha u'+v'+\overline{u'v'}}{d^\dagger q_0q_1'}\Big).
\end{align}
If $q_1q_1'\not=1$, then $S(0, d^\dagger, q, q')=0$ from the sum over $\alpha$ so that 
$$S(0, d^\dagger, q, q')\ll \mathds{1}_{q=q'}.$$  
On the other hand, if $q_1 = q_1' = 1$, the sum over $\alpha$ forces $u = u'$, and the sum over $u$ becomes a Ramanujan sum of the form
$$\sumstar_{u \bmod d^\dagger q_0q_1} e\bfrac{\bar{u}(\bar v - \bar v')}{d^\dagger q_0} = \sum_{k|(d^\dagger q_0, \bar v - \bar v')} \mu\bfrac{d^\dagger q_0}{k} k,
$$
and this leads to the crude bound $Q^\varepsilon$.

If $l \not=0$, then modulo $q_1$ the exponential sum becomes a classical Kloosterman sum which is bounded by $1/\sqrt{q_1}$. Similarly for the sum modulo $q_1'$. For $p\mid d^\dagger q_0$, the sum modulo $p$ becomes (with $d''=d^\dagger q_0/p$)
\begin{align}
	&\frac{1}{p^3}\sum_{\alpha \bmod p}\sum_{u,v,u', v' \bmod p}e\Big(\frac{\overline{q_1d''}(\bar{h}\alpha u+v+\overline{uv})-\overline{q_1'd''}(\bar{h}\alpha u'+v'+\overline{u'v'})+\overline{d''q_1q_1'}l\alpha}{p}\Big)\\
	&\ll \frac{1}{p^2}\sum_{\substack{u, v,v'\bmod p}}e\Big(\frac{\overline{q_1d''}(v+\overline{uv})-\overline{q_1'd''}(v'+q_1\overline{(q_1'u+hl)v'})}{p}\Big).
\end{align}
Recall that $(p,h)=1$. If $p\mid l$, then the above is bounded by $1$ by applying the Weil bound for the $v, v'$-sum and the trivial bound for the $u$-sum. If $p\nmid l$, we can write the above as
\begin{align}
	&\frac{1}{p^2}\sum_{\substack{x,y,z\bmod p\\(xyz,p)=1\\}}e_p\Big(y+\frac{b(x)}{y}+z+\frac{c(x)}{z}\Big)
\end{align}
where $$b(x)=\frac{1}{(q_1d'')^2x}, \quad c(x)=\frac{q_1}{(q_1'd'')^2(q_1'x+hl)}.$$
Here, we have used the change of variable
$$y=\overline{q_1d''}v, \quad x=u, \quad z=-\overline{q_1'd''}v',$$ 
so that 
$$\frac{b(x)}{y}=b(x)q_1d''\overline{v}=\overline{q_1d''uv},$$
and 
$$\frac{c(x)}{z}=-c(x)q_1'd''\overline{v'}=-q_1\overline{q_1'd''}\overline{q_1x+hl}.$$
We have \begin{align}
	\frac{b(x)}{c(x)}=\frac{q_1'^2(q_1'x+hl)}{q_1^3x},
\end{align}
which is the constant $1$ if and only if $p\mid q_1^3-q_1'^3$ and $p\mid l$. Since we are assuming that $p\nmid l$, by Bombieri–Sperber \cite[Theorem 4]{BS}, we have the sum over $x,y,z$ modulo $p$ is $\ll p^{3/2}$. Therefore, we have 
\begin{align}
	S(l, d^\dagger , q, q')\ll Q^\varepsilon\frac{(d^\dagger q_0, l)^{1/2}}{\sqrt{d^\dagger q_0q_1q_1'}}.
\end{align}

\end{proof}

By Lemma \ref{lem:Ssumbdd} we obtain that \eqref{lsum} is bounded by (noting that $q_2d'dg=u'd'd^\dagger$)
\begin{align}
	&L Q^\varepsilon \Big(\sum_{\substack{q_0\\ q_1=q_1'=1}}1+\sum_{q_0, q_1, q_1'}\sum_{0\not=|l|\ll d^\dagger q_0q_1q_1'/L} \frac{(d^\dagger q_0, l)^{1/2}}{\sqrt{d^\dagger q_0q_1q_1'}}\Big)\Psi\Big(\frac{u'd'd^\dagger q_0q_1}{Q}\Big)\Psi\Big(\frac{u'd'd^\dagger q_0q_1'}{Q}\Big)\\
	&\ll LQ^{\varepsilon}\frac{Q}{u'd'd^\dagger }+LQ^\varepsilon\sum_{q_0, q_1, q_1'}\sum_{t\mid  d^\dagger q_0}\frac{t^{1/2}}{\sqrt{d^\dagger q_0q_1q_1'}}\frac{d^\dagger q_0q_1q_1'}{tL}\Psi\Big(\frac{u'd'd^\dagger q_0q_1}{Q}\Big)\Psi\Big(\frac{u'd'd^\dagger q_0q_1'}{Q}\Big)\\
	&\ll\frac{Q^{1+\varepsilon}L}{u'd'd^\dagger }+Q^{\varepsilon}\sqrt{d^\dagger }\big(\frac{Q}{d^\dagger u' d' g}\big)^{3}.
\end{align}

From the exponential decay of $G$, the Gamma factors and \eqref{MellinHbound}, we can truncate the $w, s$ to $|w|, |s|\ll Q^\varepsilon$ with a negligible error term. It then follows that 
\begin{align}
	P_2(M, N)&\ll (Q_2Q)^{1/2+\varepsilon}\Big(\sum_{\substack{q_2\asymp Q_2\\ d\ll D\\ d^\dagger =q_2dg/u'}} \sum_{g\ll Q/q_2d'd} \sum_{u'\mid q_2}\sum_{d'\ll Q}\\&\quad\quad\quad\quad \quad\quad\quad \times \sum_{h\mid d'u'}\frac{1}{u'd'dh}\big(\frac{QL}{u'd'd^\dagger}+\frac{Q^{3}}{{d^\dagger}^{5/2}(u'd'g)^3}\big)\Big)^{1/2}\\& \ll Q^{2+\varepsilon} \Big((\frac{Q_2N}{QM})^{1/2}+\frac{1}{Q_2^{1/4}}\Big)
\end{align}
using $L\ll N Q^{1+\varepsilon}h/(u'd'M)$. 
Thus for $M\gg Q^\delta Q_2^{3/2}$ we have 
\begin{align}\label{Mlarge}
	P_2(M, N)\ll Q^{2+\varepsilon}Q_2^{-1/4}.
\end{align}
For $M\ll Q^{\delta}Q_2^{3/2}$, we apply \eqref{P2Kuznetsov} to see that 
\begin{equation}\label{Msmall0}
	\begin{split}
	P_{2}(M, N)\ll & \Big( \frac{Q^{2+\varepsilon}}{D_0^{1/2}} +Q^{1+2\theta+\delta(1-\theta)+\varepsilon}Q_2^{5/2-3\theta/2}D_0 +Q^{3/2+\delta+\varepsilon}Q_2^{5/4}D_0^{1/2}\Big(1+ \big(\frac{Q^{1-2\delta}}{Q_2^{7/2}D_0^2}\big)^\theta\Big)
	\Big).
\end{split}
\end{equation}
By setting $$ \frac{Q^{2+\varepsilon}}{D_0^{1/2}} = Q^{3/2+\delta+\varepsilon}Q_2^{5/4}D_0^{1/2},$$
we choose
$D_0=Q^{1/2-\delta}Q_2^{-5/4}$, which satisfies \eqref{D0condition}, to obtain that for $M\ll Q^\delta Q_2^{3/2}$ the upper bound
\begin{align}\label{Msmall}
	P_2(M,N)\ll &Q^\varepsilon (Q^{7/4+\delta/2}Q_2^{5/8} +Q^{3/2+(2-\delta)\theta}Q_2^{5/4-3\theta/2}+Q^{7/4+\delta/2}Q_2^{5/8}(1+Q_2^{-\theta})).
\end{align}
Combining \eqref{Mlarge} and \eqref{Msmall}, we obtain the Proposition \ref{prop:P2}.

\section{Acknowlegements}
The project began with a question Y.L. posed to X.L. after a seminar talk at KSU, and later discussed with J.L. at the goMATH mini-Symposium: Harmonic Analysis and Number Theory, at ETH Zurich in 2022. We thank the organizers for creating such a stimulating environment. Y.L. acknowledges support from the Key R\&D Program of Shandong Province, China (No. 2026ZLYS01). X.L. acknowledges support from Simons Travel Grant 962494 and NSF DMS-2302672. 
J.L. acknowleges support from Valentin Blomer during the initial stages of the project and support from NSF DMS-2502537 and the Max Planck Institute for Mathematics toward its completion. 
\bibliographystyle{plain}
\bibliography{GL3GL1biblio}




\end{document}